\documentclass{elsarticle}
\usepackage{geometry,graphicx,amssymb,amsfonts,amsmath,bm,url}
\usepackage{natbib}

\usepackage{latexsym,bm,amssymb,amsfonts,amsbsy,amsmath,graphics,graphicx,color}
\usepackage{pifont}
\usepackage{url}
\usepackage{amsmath}
\usepackage{amsfonts}
\usepackage{amssymb}

\def\bd{\begin{description}}
	\def\ed{\end{description}}

\def\beq{\begin{equation}}
\def\eeq{\end{equation}}
\def\bea{\begin{eqnarray}}
\def\eea{\end{eqnarray}}
\def\beas{\begin{eqnarray*}}
	\def\eeas{\end{eqnarray*}}

\def\RR{\mathbb R}
\def\NN{\mathbb N}

\def\pmatrix{\left(\begin{array}}
\def\endpmatrix{\end{array}\right)}

\def\go{\hbox{$\displaystyle{\mbox{\ding{172}}}$}}
\def\G1{\hbox{$\displaystyle{\mbox{\ding{172}}}$}}

\newlength\figureheight
\newlength\figurewidth

\newcommand{\be}{\begin{equation}}
\newcommand{\ee}{\end{equation}}

\def\QED{\mbox{$\Box{~}$}}

\newtheorem{Pro}{Proposition}
\newtheorem{lemma}{Lemma}

\newtheorem{theorem}{Theorem}

\newtheorem{example}{Example}
\usepackage{graphicx,color,url}

\usepackage{lipsum}
\usepackage{amsfonts}
\usepackage{graphicx}
\usepackage{epstopdf}
\usepackage{algorithmic}
\ifpdf
\DeclareGraphicsExtensions{.png,.eps,.pdf,.jpg}
\else
\DeclareGraphicsExtensions{.eps}
\fi

\usepackage{amsopn}

\journal{Applied Numerical Mathematics}

\begin{document}

\begin{frontmatter}

\title{Conjugate-symplecticity properties of Euler--Maclaurin methods and their implementation on the Infinity Computer\tnoteref{t1}}

\tnotetext[t1]{This work was funded by the INdAM-GNCS 2018 Research Project
			``Numerical methods in optimization and ODEs'' (the authors are
			members of the INdAM Research group GNCS). Research of  F. Iavernaro and F. Mazzia was also supported by
			the Universit\`a degli Studi di Bari, project
			``Equazioni di Evoluzione: analisi qualitativa e metodi numerici''.}

\author[felix]{F. Iavernaro}
\ead{felice.iavernaro@uniba.it}
\author[frances]{F. Mazzia\corref{cor1}}
\ead{francesca.mazzia@uniba.it}
\cortext[cor1]{Corresponding author}
\author[maratyaro1,maratyaro2]{M.S.~Mukhametzhanov}
\ead{m.mukhametzhanov@dimes.unical.it}
\author[maratyaro1,maratyaro2]{Ya.D.~Sergeyev}
\ead{yaro@dimes.unical.it}

\address[felix]{Dipartimento di Matematica, Universit\`a degli Studi di Bari Aldo Moro, Italy}
\address[frances]{Dipartimento di Informatica, Universit\`a degli Studi di Bari Aldo Moro, Italy}
\address[maratyaro1]{DIMES, Universit\`a della  Calabria, Italy}
\address[maratyaro2]{Department of Software and Supercomputing
		Technologies Lobachevsky State University of Nizhni Novgorod,
		Russia}
%\address[yaro]{DIMES, Universit\`a della  Calabria, Italy, Department of Software and Supercomputing
%		Technologies Lobachevsky State University of Nizhni Novgorod,
%		Russia}

\begin{abstract}
 Multi-derivative one-step methods based upon  Euler--Maclaurin integration formulae are considered for the solution of canonical Hamiltonian dynamical systems. Despite the negative result that simplecticity may not be attained by any multi-derivative Runge--Kutta methods, we show that the  Euler--MacLaurin method of order $p$ is conjugate-symplectic up to order $p+2$. This feature entitles them to play a  role in the context of geometric integration and, to make their implementation competitive with the existing integrators,  we  explore the possibility of computing the underlying higher order derivatives with the aid of the Infinity Computer.
 
\end{abstract}

\end{frontmatter}

\section{Introduction}
\label{sec:intro}
In the present work, we will consider the application of multi-derivative one-step methods for the numerical solution of canonical Hamiltonian problems
\begin{equation}\label{ham}
y' = J\nabla H(y), \qquad y(t_0)=y_0\in\RR^{2m},
\end{equation}
with
\begin{equation}\label{yJ}
y = \pmatrix{c} q\\ p\endpmatrix,\quad q,p\in\RR^m, \qquad J = \pmatrix{rr} O &I\\ -I &O\endpmatrix,
\end{equation}
where $q$ and $p$ are the generalized coordinates and conjugate momenta, \\$H:\RR^{2m} \rightarrow \RR$ is the Hamiltonian function and $I$ stands for the identity matrix of dimension $m$.
It is well-known that the flow $\varphi_t: y_0 \rightarrow y(t)$ associated with the dynamical system (\ref{ham}) is symplectic, namely its Jacobian satisfies
\begin{equation}
\label{symflow}
\frac{\partial \varphi_t(y)^\top}{\partial y} J \frac{\partial \varphi_t(y)}{\partial y} =J, \quad \mbox{for all }\, y\in\RR^{2m}.
\end{equation}
Symplecticity is a characterizing property of canonical Hamiltonian systems and has relevant implications on the geometric properties  of the orbits in the phase space. Consequently, the search of symplectic methods for the numerical integration of (\ref{ham})  forms a prominent branch of research. We recall that a one-step method $y_1=\Phi_h(y_0)$ ($h$ is the stepsize of integration) is called symplectic if its Jacobian matrix is symplectic, i.e. $\Phi_h$ satisfies the analog of (\ref{symflow}) with $\Phi_h(y)$ in place of $\varphi_t(y)$.  One prominent feature of  symplectic integrators is, by definition, the conservation of the symplectic differential 2-form associated with matrix $J$ in (\ref{yJ}) which, in turn, implies the conservation of all quadratic first integrals of a Hamiltonian system. Though they fail to conserve non quadratic Hamiltonian functions, a backward error analysis shows that, when implemented with constant stepsize and under regularity assumptions, they provide a near conservation of the Hamiltonian function over exponentially long times \cite{BG94} (see also \cite[page 366]{HLW06}).

The study of symplecticity in combination with multi-derivative R-K methods was initiated by  Lasagni \cite{La90} who provided a sufficient algebraic condition for a multi-derivative Runge--Kutta method to be symplectic. The brief investigation culminated with the work of Hairer, Murua, and Sanz Serna \cite{HaMuSS94} who showed that, for irreducible multi-derivative R--K methods, Lasagni's condition is also necessary but it may only be satisfied by standard R--K formulae.

{ Given this background, it does make sense to wonder whether one-step multi-derivative formulae may share some weaker conditions related to symplecticity. A method $y_1=\Phi_h(y_0)$ is conjugate to a symplectic method $y_1=\Psi_h(y_0)$ if a global change of coordinates $\chi_h(y)=y+O(h)$ exists such that $\Phi_h=\chi_h \circ \Psi_h \circ \chi_h^{-1}$. We observe that the  solution $\{y_n\}$ of a symplectic conjugate method satisfies
$y_n=\Phi^n_h(y_0)=(\chi_h \circ \Psi_h \circ \chi_h^{-1})^n(y_0)=\chi_h \circ \Psi_h^n \circ \chi_h^{-1}(y_0)$.
Consequently,  symplectic conjugate methods  inherit the long-time behavior of symplectic integrators. In the present work we are interested in a generalization of the conjugate-symplecticity property, introduced in \cite{HZ12}. A method $y_1=\Phi_h(y_0)$ of order $p$ is conjugate-symplectic  up to order $p+r$, with $r\ge 0$, if a global change of coordinates $\chi_h(y)=y+O(h^p)$ exists such that $\Phi_h=\chi_h \circ \Psi_h \circ \chi_h^{-1}$, with the map $\Psi_h$ satisfying 
\begin{equation}
\label{csup}
\Psi_h'(y)^TJ\Psi_h'(y)=J+O(h^{p+r+1}). 
\end{equation}
A consequence of property (\ref{csup}) is that the method $\Phi_h(y)$ nearly conserves all quadratic first integrals and the Hamiltonian function over time intervals of length $O(h^{-r})$ (see \cite{HZ12}).
}

This path of investigation is further motivated by the recent studies concerning the implementation of methods involving higher derivatives of the vector field on the Infinity Computer, {a new type of a supercomputer allowing one to work \emph{numerically} with infinite and infinitesimal numbers} \cite{AmIaMaMuSe16,MaSeIaAmMu16,Se13,Se15,SeMuMaIaAm16}. The final goal of this new approach is to improve the computational effort associated with the evaluation of the involved derivatives and make them competitive with more standard integrators. 
In this paper, the Infinity Computer is used for this purpose. It is based on the positional numeral system with the infinite radix $\G1$ (called \emph{grossone} and introduced as the number of elements of the set of natural numbers $\NN$) introduced  in \cite{Sergeyev,informatica,Lagrange} (see also recent surveys
\cite{UMI,EMS}). The first ideas that can be considered as predecessors to the Infinity Computing and based on the principle ``the part is less than the whole'' were studied by Bernard Bolzano (see \cite{bolzano} and a detailed analysis in \cite{Trlifajova}). It should be noted that the Infinity Computing theory is not related either to Cantor's cardinals and ordinals (\cite{Cantor})  or non-standard analysis and Levi-Civita field (\cite{Robinson,Levi-civita,Se19a}).

In the Infinity Computing, with the introduction
of \G1 in the mathematical language, all other symbols (like
$\infty$, Cantor's $\omega$, $\aleph_0, \aleph_1, \dots$,  etc.)
traditionally used to deal  with infinities and infinitesimals in different situations are
excluded from the language, because \G1 and other numbers constructed
with its help not only can replace all of them but can be
used with a higher accuracy. The \G1-based numeral system avoids indeterminate forms and situations similar to   $\infty + 1= \infty$
and $\infty - 1= \infty$ providing results ensuring that if $a$ is a
numeral written in this numeral system then for any $a$ (i.e., $a$
can be finite, infinite, or infinitesimal)   it follows $a+1>a$ and
$a-1<a$.

To construct a number $C$ in the \G1-based numeral system, we subdivide $C$
into groups corresponding to powers of \ding{172}:
\beq
C = c_{p_{m}}
\G1^{p_{m}} +  \ldots + c_{p_{1}} \G1^{p_{1}} +c_{p_{0}} \G1^{p_{0}}
+ c_{p_{-1}} \G1^{p_{-1}} + \ldots   + c_{p_{-k}}
\G1^{p_{-k}}.
\label{numberc}
\eeq
Then, we can write down the number $C$ as follows:
\beq
C = c_{p_{m}}
\G1^{p_{m}}    \ldots   c_{p_{1}} \G1^{p_{1}} c_{p_{0}} \G1^{p_{0}}
c_{p_{-1}} \G1^{p_{-1}} \ldots c_{p_{-k}}
\G1^{p_{-k}},
\label{grossnumber}
\eeq
where all numerals $c_i\neq0$
belong to a traditional numeral system and are called
\textit{grossdigits}, while numerals $p_i$ are  sorted in the decreasing order
with $ p_0=0$
\[
p_{m} >  p_{m-1}  > \ldots    > p_{1} > p_0 > p_{-1}  > \ldots
p_{-(k-1)}  >   p_{-k},
\]
and called \textit{grosspowers}.

The term having $p_0=0$ represents the finite part of $C$ since $c_0
\G1^0=c_0$. Terms having finite positive gross\-powers represent the
simplest infinite parts of~$C$. Analogously, terms having   negative
finite grosspowers represent the simplest infinitesimal parts of
$C$. For instance, the simplest infinitesimal used in this work as the integration step in the Euler method for computing the derivatives is  $\G1^{-1}=\frac{1}{\G1}$.

The \G1-based methodology   has   been successfully applied in
several areas of Mathematics and Computer Science: single and
multiple criteria optimization (see
\cite{Cococcioni,DeLeone,DeLeone_2}), handling ill-con\-di\-tio\-ning (see \cite{Gaudioso&Giallombardo&Mukhametzhanov(2018),
	homogeneity}),  cellular automata (see \cite{DAlotto,DAlotto_3}),
Euclidean and hyperbolic geometry (see \cite{Margenstern_3}),
percolation (see \cite{Iudin_2,DeBartolo}), fractals (see
\cite{Caldarola_1,Koch,DeBartolo}),  infinite
series and the Riemann zeta function (see
\cite{UMI,EMS,Zhigljavsky}), the first Hilbert
problem and supertasks (see \cite{Rizza,first,EMS,Rizza19a}), Turing machines
and probability (see \cite{Rizza_2,EMS,Sergeyev_Garro,Sergeyev_Garro_2}), etc.

An interesting peculiarity of the Infinity Arithmetic methodology in the context of this paper is that it allows one to work with black-box functions, namely the analytical expression of the function $f(y)$ may be unknown. In other  words, the function  $f$ can be given by a code or formula which are unknown to the user. He/she provides an argument $y$ and obtains a result $f(y)$ without any knowledge about how this result has been obtained. In particular, this means that the user can not calculate exact derivatives either analytically or symbolically. It has been shown that the Infinity Arithmetic methodology can successfully handle this situation in the context of numerical differentiation and   solution of  ordinary differential equations
(see \cite{AmIaMaMuSe16,Se13,SeMuMaIaAm16,Num_dif}).

The paper is organized as follows. In Section \ref{sec:EulMac} we introduce Euler--Maclaurin methods while in Section \ref{sec:consym} we show their  conjugate symplecticity properties. Section \ref{sec:inf} is devoted to the efficient computation of the derivatives on the Infinity Computer. Finally, some numerical illustrations are presented in Section \ref{sec:num} while Section \ref{sec:conc} contains some concluding remarks.

\section{Euler--MacLaurin methods}
\label{sec:EulMac}
Euler--Maclaurin methods are higher derivative collocation methods belonging to the class of Hermite-Obrechkov methods \cite[page 277]{HNW93}. When applied to the general initial value problem
\begin{equation}
\label{ivp}
y'(t)=f(y(t)), \quad y(t_0)=y_0 \in \RR^m,
\end{equation}
they yield a polynomial $\sigma(t_0+ch)$ approximating the true solution $y(t)$ in the interval $[t_0,t_0+h]$ ($h$ is the stepsize of integration)  defined by means of the following collocation conditions at the ends of the interval
\begin{equation}
\label{colloc}
\left\{
\begin{array}{l}
\sigma(t_0)=y_0, \\[.2cm]
\sigma^{(j)}(t_0) = { D_{j-1} f}(\sigma(t_0)), \quad j=1,\dots,s,\\[.2cm]
\sigma^{(j)}(t_0+h) = { D_{j-1} f}(\sigma(t_0+h)), \quad j=1,\dots,s.
\end{array}
\right.
\end{equation}
where, for any given vector $z$, $D_jf(z)$ denotes the total $j$-th time derivative of $f(y(t))$ evaluated at $y(t)=z$, with $y(t)$ formally satisfying (\ref{ivp}):
\begin{equation}
\label{ftimeder}
D_jf(z)= \left.\frac{\mathrm{d}^j}{\mathrm{d}t^j}f(y(t))\right|_{y(t)=z}. 
\end{equation}
We have used here the subscript $j$  to distinguish the operator defined in (\ref{ftimeder}) (Lie derivative) from the classical time-derivative operator of order $j$  which will be denoted, as usual, by $D^j$.    
The approximation at time $t_1=t_0+h$ is then yielded by $y_1=\sigma(t_0+h)\simeq y(t_0+h)+O(h^{p+1})$, with $p=2s$.
{
Notice that the right-hand side of (\ref{ftimeder}) may be expressed in terms of $y(t)$ via the relation (\ref{ivp}). 
 For example, for a given $z\in \RR^m$, 
$$
D_1f(z)=\left.f'(y(t))y'(t)\right|_{y(t)=z} = \left.f'(y(t))f(y(t))\right|_{y(t)=z} = f'(z)f(z),
$$ 
where $f'$ denotes the Jacobian matrix of the function $f$. More in general, the analytical computation of $D_jf(z)=D_1(D_{j-1}f(z))$  involves a tensor of order $j+1$.} This considerably raises the computational cost associated with the implementation of the method as long as higher derivatives are considered.   We will see that the use of the Infinity Computer circumvents this issue by producing a precise value of { $D_jf(z)$} without explicitly evaluating its analytical expression in terms of the derivatives of $f$.

These integrators derive their name from the well-known Euler--Maclaurin integration formula: if $m$ and $n$ are natural numbers and $g(x)$ with $x\in \RR$ is a regular function defined on $[m,n]$,
\begin{equation}
\label{EM}
\begin{array}{rcl}
\displaystyle \int_m^n g(x)\mathrm{d} x &=& \displaystyle  \frac{g(m)+g(n)}{2}+\sum_{i=m+1}^{n-1}g(i) \\[.3cm]
&& \displaystyle  -\sum_{k=1}^{s-1} \frac{B_{2k}}{(2k)!} \left( g^{(2k-1)}(n) - g^{(2k-1)}(m)\right)+R,
\end{array}
\end{equation}
where $s\ge 1$ and $B_{2k}$ is the $2k$-th Bernoulli number.\footnote{In formula (\ref{EM}) we have used $s-1$ in place of $s$ to make the argument consistent with formulae (\ref{colloc}). When $s=1$ we get the standard composite trapezoidal quadrature rule. The first even Bernoulli number are $B_2 = 1/6$, $B_4 = -1/30$, $B_6 = 1/42$, $B_8 = -1/30,\dots$.} The remainder $R$ is bounded by
$$
|R| \le \frac{2}{(2 \pi)^{2s-2}} \int_m^n \left| g^{(2s-1)}(x) \right| \mathrm{d}x.
$$
We now consider the integral form of (\ref{ivp}) in the interval $[t_0,t_0+h]$, namely
$$
y(t_0+ch)=y(t_0)+h \int_0^cy'(t_0+\tau h)\mathrm{d} \tau \equiv y(t_0)+ h\int_0^cf(y(t_0+\tau h))\mathrm{d} \tau .
$$
Setting $c=1$ and evaluating the integral by means of (\ref{EM}) with $m=0$ and $n=1$ yields
$$
\begin{array}{rcl}
y(t_1)&=&\displaystyle  y(t_0)+h\int_0^1 y'(t_0+\tau h)\mathrm{d} \tau \\[.3cm]
& =& \displaystyle  y(t_0)+\frac{h}{2}\left( f(y(t_1))+f(y(t_0)) \right) -\sum_{k=1}^{s-1} \frac{h^{2k}B_{2k}}{(2k)!}  \left( {D^{2k-1}f(y(t_1))-D^{2k-1}f(y(t_0)) } \right)+R.
\end{array}
$$
An approximation $y_1\simeq y(t_1)$ is obtained by neglecting the remainder term $R=O(h^{2s+1})$. Taking into account (\ref{ftimeder}) and considering that $y(t_0)=y_0$, we arrive at the one-step Euler-Maclaurin family of methods
\begin{equation}
\label{EMd}
\begin{array}{rcl}
y_1&=&\displaystyle  y_0+\frac{h}{2}\left( f(y_1)+f(y_0) \right) -\sum_{k=1}^{s-1} \frac{h^{2k}B_{2k}}{(2k)!} \left( { D_{2k-1}f(y_1)- D_{2k-1}f(y_0) }\right).
\end{array}
\end{equation}
When $s=1$, (\ref{EMd}) becomes the trapezoidal method while, for $s=2$ and $s=3$ we get the fourth and sixth order methods
\begin{equation}
\label{EM4}
\begin{array}{rcl}
y_1&=&\displaystyle  y_0+\frac{h}{2}\left( f(y_1)+f(y_0) \right) -\frac{h^2}{12}\left( y''_1 - y''_0\right),
\end{array}
\end{equation}
\begin{equation}
\label{EM6}
\begin{array}{rcl}
y_1&=&\displaystyle  y_0+\frac{h}{2}\left( f(y_1)+f(y_0) \right) -\frac{h^2}{12}\left( y''_1- y''_0\right)+\frac{h^4}{720}\left( y^{(4)}_1 - y^{(4)}_0\right),
\end{array}
\end{equation}
where, to simplify the notation, we have set $y^{(2k)}_{i}= {D_{2k-1}f(y_{i})}$, $i=0,1$.

\section{Conjugate symplecticity properties}
\label{sec:consym}
To prove that the Euler-Maclaurin method (\ref{EMd}) is conjugate to a symplectic method up to order $2s+2$, we show that the map $y_1=\Psi_h(y_0)$  associated with (\ref{EMd}) is such that $\Psi_h(y)=\Phi_h(y)+O(h^{2s+3})$, where $y_1=\Phi_h(y_0)$ is a suitable B-series integrator satisfying property (a) of the following lemma.
\begin{lemma}\cite{CFM06}
	\label{CScond}
	Assume that problem (\ref{ivp}) admits a quadratic first integral $Q(y)=y^\top S y$ (with $S$ a symmetric matrix)  and is solved by a $B$-series integrator $\Phi_h(y)$. The following properties are equivalent:
	\begin{enumerate}
		\item[(a)] $\Phi_h(y)$ has a modified first integral of the form $\widetilde Q(y)= Q(y)+O(h)$;
		\item[(b)]  $\Phi_h(y)$ is formally conjugate to a symplectic $B$-series method.
	\end{enumerate}
\end{lemma}
Here, \emph{formally conjugate} means that the power series yielding the conjucagy needs not be convergent.
This result was stated in \cite{CFM06} and subsequently used in \cite{Ha08} to derive the conjugate symplecticity property of symmetric multistep methods.

%Guided by the above equivalence,    
%%%%%%%%%%%
%we first derive the $B$-series expansion of Euler--MacLaurin methods. Its existence may be directly deduced from \cite{HaMuSS94}, where
%a $B$-series representation of a generic multi-derivative Runge-Kutta method has been obtained. Nevertheless, we prefer to derive it by a direct computation on formula (\ref{EMrs}) which, on one hand, will be later exploited to achieve our main result and, on the other hand, will reveal a close relationship of these formulae with the trapezoidal method.
In terms of the characteristic polynomials
$$
\rho(z)=z-1, \qquad \sigma(z)=\frac{1}{2}(z+1)
$$
formula (\ref{EMd}) reads, for a generic time $t_n$,
\begin{equation}
\label{EMrs}
\rho(E)y_n = h\sigma(E)f(y_n)-\sum_{k=1}^{s-1} \frac{h^{2k}B_{2k}}{(2k)!}\rho(E) { D_{2k-1}f(y_n)},
\end{equation}
where $E$ denotes the shift operator: $E(y_n)=y_{n+1}$.

For an analytic function $u(t)$ and a stepsize $h>0$,  we introduce the shift operator $E_h(u(t))=u(t+h)$ and recall the relation 
\begin{equation}
\label{ED}
E_h=e^{hD}=\sum_{k=0}^\infty \frac{h^k}{k!}D^k.
\end{equation}
\begin{theorem}
	\label{main_result}
	The map  { $y_1=\Psi_h(y_0)$} associated with the one-step method (\ref{EMrs}) admits a $B$-series expansion and is conjugate-symplectic up to order $2s+2$.
\end{theorem}
\begin{proof}
The existence of a $B$-series expansion for $y_1=\Psi_h(y_0)$ may be directly deduced from \cite{HaMuSS94}, where
a $B$-series representation of a generic multi-derivative Runge-Kutta method has been obtained.
From the generating function of Bernoulli numbers (see, for example \cite{HaWa08})
	$$
	\frac{z}{e^z-1} = \sum_{k=0}^\infty B_k \frac{z^k}{k!},
	$$
	we get, considering that $B_1=-1/2$ and $B_k=0$ for  $k$ odd,
	\begin{equation}
	\label{bergen1}
	\frac{z\sigma(e^z)}{\rho(e^z)}= \frac{1}{2}\frac{z(e^z+1)}{e^z-1} = \frac{z}{e^z-1} +\frac{z}{2}= 1+ \sum_{k=1}^\infty B_{2k} \frac{z^{2k}}{(2k)!}.
	\end{equation}
	In the spirit of backward analysis, we look for an analytical function $u(t)$ formally satisfying the difference equation (\ref{EMrs}) that is, by virtue of (\ref{ED}),
	$$
	\rho(e^{hD})u(t) = h\sigma(e^{hD})f(u(t))-\sum_{k=1}^{s-1} \frac{B_{2k}}{(2k)!}h^{2k}\rho(e^{hD})D_{2k-1} f(u(t)).
	$$
	Multiplying both sides of the previous equation by $D\rho(e^{hD})^{-1}$ yields
	$$
	\dot u(t) = hD\rho(e^{hD})^{-1}\sigma(e^{hD})f(u(t))-\sum_{k=1}^{s-1} \frac{B_{2k}}{(2k)!}h^{2k}DD_{2k-1} f(u(t)),
	$$
	and, by taking into account (\ref{bergen1}), we finally arrive at
	\begin{equation}
	\label{mod_eqEM}
	\dot u(t) = \left(1+ \sum_{k=s}^{\infty} \frac{B_{2k}}{(2k)!}h^{2k}D^{2k}\right) f(u(t)) +\sum_{k=1}^{s-1} \frac{B_{2k}}{(2k)!}h^{2k}D(D^{2k-1}-D_{2k-1}) f(u(t)).
	\end{equation}
	Equation (\ref{mod_eqEM}) coupled with the initial condition $u(t_0)=y_0$ is nothing but the {\em modified differential equation} associated with the Euler--MacLaurin method of order $2s$, so that $u(t_0+nh)=y_n$.  

Since $u(t)=y(t)+O(h^{2s})$, we see that $(D^{2k-1}-D_{2k-1}) f(u(t))=O(h^{2s})$ and hence the solution $u(t)$ of (\ref{mod_eqEM}) is $O(h^{2s+2})$-close to the solution of the following initial value problem
\begin{equation}
	\label{mod_eq1}
	\dot u(t) = \left(1+ \sum_{k=s}^{\infty} \frac{B_{2k}}{(2k)!}h^{2k}D^{2k}\right) f(u(t)), \qquad u(t_0)=y_0.
\end{equation}
We may interpret (\ref{mod_eq1}) as the modified equation of a one-step method $y_1=\Phi_h(y_0)$, where $\Phi_h$ is evidently the time-$h$ flow associated with (\ref{mod_eq1}). Expanding the solution of  (\ref{mod_eq1}) in Taylor series, we get
	$$
	\begin{array}{rcl}
	\Phi_h(y_0)&=& \displaystyle y_1=u(t_0+h)=y_0+hf(y_0)+ \sum_{k=s}^{\infty} \frac{B_{2k}}{(2k)!}h^{2k+1}D^{2k}f(y_0)\\[.4cm]
	&& \displaystyle  + \frac{h^2}{2!} f'(y_0)f(y_0)+ \sum_{k=s}^{\infty} \frac{B_{2k}}{(2k)!}h^{2k+2}D^{2k+1} f(y_0) +\dots .
	\end{array}
	$$
	where $f'(y)$ is the Jacobian of $f(y)$ and we have set $D^{r}f(y_0)=\left.D^{r}f(u(t))\right|_{t=t_0}$. Collecting like powers of $h$ in the above expression yields a formal power series expansion in the stepsize $h$, that is a $B$-series expansion.
	
	To show that $\Phi_h(y)$ admits a modified first integral $\widetilde Q(y)=Q(y)+O(h^{2s})$, we follow the same flow of computation appearing in \cite[Theorem 4.10 on page 591]{HLW06}, that states an analogous property for symmetric linear multistep methods. We first notice that
	$$
	\left(1+ \sum_{k=s}^{\infty} \frac{B_{2k}}{(2k)!}z^{2k}\right)^{-1}= 1+ \sum_{k=s}^{\infty} \gamma_kz^{2k},
	$$
	for suitable coefficients $\gamma_k$. Thus (\ref{mod_eq1}) is tantamount to
	\begin{equation}
	\label{mod_eq2}
	\left(1+ \sum_{k=s}^{\infty} \gamma_k h^{2k}D^{2k}\right) \dot u(t) =  f(u(t)).
	\end{equation}
	Multiplying both sides of (\ref{mod_eq2}) by the term $u(t)^\top S$ yields
	\begin{equation}
	\label{mod_eq3}
	\frac{1}{2}\frac{\mathrm{d}}{\mathrm{d}t}Q(u(t))+ \sum_{k=s}^{\infty} \gamma_k h^{2k}u(t)^\top S u^{(2k+1)}(t) =  0,
	\end{equation}
	where we have taken into account that $2u(t)^\top S\dot u(t)=\dot Q(u(t))$ and $z^\top S f(z)=0$ for any $z\in \RR^m$, since $Q(y)$ is a first integral of the original system (\ref{ivp}).   A repeated use of the property (the explicit dependence on the time $t$ is omitted to simplify the notation)
	$$
	{u^{(i)}}^\top S u^{(j)}=\frac{\mathrm{d}}{\mathrm{d}t} \left( {u^{(i)}}^\top S u^{(j-1)}\right)-{u^{(i+1)}}^\top S u^{(j-1)},
	$$
	with, in particular,
	$$
	{u^{(i)}}^\top S u^{(i+1)}=\frac{1}{2}\frac{\mathrm{d}}{\mathrm{d}t} \left( {u^{(i)}}^\top S u^{(i)}\right),
	$$
	allows us to cast each term $u(t)^\top S u^{(2k+1)}(t)$ in (\ref{mod_eq3}) as
	$$
	u^\top S u^{(2k+1)}=\frac{\mathrm{d}}{\mathrm{d}t}\left(u^\top S u^{(2k)} - \dot u^\top S u^{(2k-1)}+\dots+(-1)^k\frac{1}{2}{u^{(k)}}^\top S u^{(k)}  \right).
	$$
	We observe that the sum in brackets on the right hand side may be formally cast as a function of $u(t)$ by replacing all the derivatives with the aid of the modified differential equation (\ref{mod_eq1}). After this substitution, denoting by
	$$
	Q_k(u(t)) = 2 \gamma_k u(t)^\top S u^{(2k+1)}(t)
	$$
	and $\widetilde Q(u)= Q(u)+\sum_{k=s}^{\infty} h^{2k} Q_k(u)$, from (\ref{mod_eq3}) we finally obtain $\frac{\mathrm{d}}{\mathrm{d}t} \widetilde Q(u(t))=0$ which concludes the proof, {since $\Psi_h(y)=\Phi_h(y)+O(h^{2s+3})$.} \QED
\end{proof}

\section{Computation of the derivatives} 
\label{sec:inf}
%\marginpar{\tiny \color{blue} Abbiamo riallineato le formule centrate di questa sezione e introdotto qualche piccola modifica in alcune parti del testo (non evidenziata in blu).}
One drawback with these implicit methods is the computation of  high order derivatives. Symbolic or automatic differentiation are often preferred to finite differences techniques involving terms in $y$ and $y'$, which suffer from numerical instability when the increment becomes small. %, and requires additional stages. %the behavior of the resulting method will suffer from lack of high accuracy in  these approximations due to ill-conditioning issues}. 
This  drawback is overcome on the Infinity Computer and hereafter we illustrate two possible approaches in order to compute the $k$-th derivative of $y(t)$ at time $t_i$.

\smallskip

\textbf{Strategy (a)}.  This strategy was first proposed in \cite{Se13}. 
We perform $k$ infinitesimal steps starting at 
time $t_i$ using the explicit Euler formula with stepsize $\bar h=\G1^{-1}$ as follows:
$$
y_{i,1} = y_i + \go^{-1}f(y_i),\quad  y_{i,2} = y_{i,1} + \go^{-1}f(y_{i,1}),\dots,\quad 
y_{i,k} = y_{i,k-1} + \go^{-1}f(y_{i,k-1}).
$$
Then,  the values of the needed derivatives can be obtained by means of the
forward dif\-feren\-ces $F_{\bar h}^k[y_{i,0}, y_{i,1},\dots, y_{i,k}]$, with $\bar h =
\go^{-1}$ as follows
\begin{equation}
\label{S1}
y^{(k)}(t_i) =
\frac{F_{\mbox{\tiny{\ding{172}}}^{-1}}^k[y_{i,0}, y_{i,1},\dots, y_{i,k}]}{\go^{-k}} +
O(\go^{-1})
\end{equation}
where
\begin{equation}
\label{Delta1}
F_{\mbox{\tiny{\ding{172}}}^{-1}}^k[y_{i,0}, y_{i,1},\dots, y_{i,k}] =
\sum_{j=0}^k (-1)^j {k \choose j}{y_{i,k-j}},~~y_{i,0}=y_i. \quad
\end{equation}
As was proven in \cite{Se13}, since  the error of the approximation is $O(\go^{-1})$, the finite part of the value
$\frac{F_{\mbox{\tiny{\ding{172}}}^{-1}}^k[y_{i,0}, y_{i,1},\dots, y_{i,k}]}{\go^{-k}}$   gives
the \textit{exact} derivative $y^{(k)}(t_i)$. For a more detailed
description of the numerical computation  of exact derivatives on
the Infinity Computer, see \cite{Se13}.
\smallskip

\textbf{Strategy (b)}.  Let us propose another strategy for computing the exact derivatives, where finite differences may be employed directly on the value of $f$ as follows:
\begin{equation}
\label{S2}
y^{(k)}(t_i) = D^{k-1} f(y_i) =
\frac{F_{\mbox{\tiny{\ding{172}}}^{-1}}^{k-1}[f(y_{i,0}), f(y_{i,1}),\dots, f(y_{i,k-1})]}{\go^{-(k-1)}} +
O(\go^{-1})
\end{equation}

Now, let us prove that  formulae (\ref{S1}) and (\ref{S2}) are equivalent.
\begin{Pro}
	Let us suppose that for the solution $y(t)$ of the ordinary differential equation (\ref{ivp}) it is known the value $y_i = y(t_i)$ at the point $t_i$. Then  formulae (\ref{S1}) and (\ref{S2}) are equivalent.
\end{Pro}
\begin{proof}
	Since  formulae (\ref{S1}) and (\ref{S2}) differ only in the forward differences, then in order to prove the proposition, it is sufficient to demonstrate that the following equality holds true
	\begin{equation}
	\label{forward_equation}
	F_{\mbox{\tiny{\ding{172}}}^{-1}}^k[y_{i,0}, y_{i,1},\dots, y_{i,k}] = \G1^{-1} \cdot F_{\mbox{\tiny{\ding{172}}}^{-1}}^{k-1}[f(y_{i,0}),\dots, f(y_{i,k-1})],
	\end{equation}
	where $y_{i,0} = y_i$. Let us use the mathematical induction to prove it. For $k=1$, the assertion is trivial. However, since the first  meaningful case is for $k=2$ and in order to make the reader more acquainted with grossone based formalism, we consider this case in full detail.  By using  formulae (\ref{S1})--(\ref{S2}) for $y_{i,0},~y_{i,1}$, and $y_{i,2}$ we obtain
	
	\begin{eqnarray*}
		F_{\mbox{\tiny{\ding{172}}}^{-1}}^2[y_{i,0}, y_{i,1}, y_{i,2}] &=& y_{i,2}-2y_{i,1}+y_{i,0}\\
		&=& y_{i,0} + \G1^{-1} \cdot (f(y_{i,0})+f(y_{i,1})) - 2(y_{i,0} + \G1^{-1}) \cdot f(y_{i,0}) + y_{i,0} \\
		&=& \G1^{-1} \cdot (f(y_{i,1})-f(y_{i,0})) = \G1^{-1} \cdot F_{\mbox{\tiny{\ding{172}}}^{-1}}^1[f(y_{i,0}), f(y_{i,1})].
	\end{eqnarray*}
	
	Suppose now that (\ref{forward_equation}) holds for $k-1,~k\geq 3$. We get
	\begin{eqnarray*}
		F_{\mbox{\tiny{\ding{172}}}^{-1}}^k[y_{i,0},\dots, y_{i,k}] &=& F_{\mbox{\tiny{\ding{172}}}^{-1}}^{k-1}[y_{i,1}, \dots, y_{i,k}] - F_{\mbox{\tiny{\ding{172}}}^{-1}}^{k-1}[y_{i,0}, \dots, y_{i,k-1}]  \\
		&=&  \G1^{-1} \cdot F_{\mbox{\tiny{\ding{172}}}^{-1}}^{k-2}[f(y_{i,1}), \dots, f(y_{i,k-1})] \\
		&& - \G1^{-1} \cdot F_{\mbox{\tiny{\ding{172}}}^{-1}}^{k-2}[f(y_{i,0}), \dots, f(y_{i,k-2})] \\
		&=& \G1^{-1} \cdot F_{\mbox{\tiny{\ding{172}}}^{-1}}^{k-1}[f(y_{i,0}), \dots, f(y_{i,k-1})]. 
	\end{eqnarray*}
	
	%=  \G1^{-1} \cdot F_{\mbox{\tiny{\ding{172}}}^{-1}}^{k-2}[f(y_{i,1}), \dots, f(y_{i,k-1})] - $$
	% $$- \G1^{-1} \cdot F_{\mbox{\tiny{\ding{172}}}^{-1}}^{k-2}[f(y_{i,0}), \dots, f(y_{i,k-2})] = $$
	% $$= \G1^{-1} \cdot F_{\mbox{\tiny{\ding{172}}}^{-1}}^{k-1}[f(y_{i,0}), \dots, f(y_{i,k-1})].$$
	
	This completes the proof. \QED
\end{proof}

The advantage of strategy (b) with respect to strategy (a) is that the use of  formula (\ref{S2}) in place of (\ref{S1}) decreases the computational costs due to the following reasons:
\begin{description}
	\item 1. It is not necessary to compute the value $y_{i,k}$ for (\ref{S2}), whereas it should be calculated in (\ref{S1}).
	\item 2. All the computations using (\ref{S1}) should be performed using the grosspowers up to $-k$. In contrast,  formula (\ref{S2}) allows us to work with the numbers using only the grosspowers up to $-(k-1)$.
\end{description}
Let us consider the following example from \cite{SeMuMaIaAm16} in order to illustrate these issues.
\begin{example}
	Let us find the first 3 derivatives $y'(t_0)$, $y''(t_0)$, and $y'''(t_0)$ of the solution $y(t)$ at the point $t_0 = 0$ of the following initial value problem:
	\begin{equation}
	\label{example1}
	\frac{dy}{dt} = \frac{y-2ty^2}{1+t},~y(t_0) = 0.4,
	\end{equation}
	whose exact solution is 
	\begin{equation}
	\label{sol1}
	y(t) = \frac{1+t}{2.5+t^2}.
	\end{equation}
	Differentiating (\ref{sol1}) we get the exact values of the following derivatives:
	$$y'(t_0)=0.4, \quad y''(t_0)=-0.32, \quad y'''(t_0)=-0.96.$$
	Now, let us find these derivatives using strategy (a). First, we perform 3 iterations of the Euler method with the integration step $\bar h = \G1^{-1}$, truncating all values after the grosspower $-3$:
	\begin{eqnarray*}
		y_{1} &=& y_0 + \go^{-1}f(t_0,y_0) = 0.4 + 0.4\go^{-1},\\
		y_2 &=& y_1 + \go^{-1}f(t_0+\go^{-1},y_1)=0.4+0.8\go^{-1}-0.32\go^{-2}-0.32\go^{-3},\\
		y_3 &=& y_2 + \go^{-1}f(t_0+2\go^{-1},y_2)=0.4+1.2\go^{-1}-0.96\go^{-2}-1.92\go^{-3}.
	\end{eqnarray*}
	Applying  formulae (\ref{S1}), (\ref{Delta1}), we obtain
	\begin{eqnarray*}
		y'(t_0) &\simeq& \G1 \cdot F_{\mbox{\tiny{\ding{172}}}^{-1}}^1[y_{0}, y_{1}] = \G1\cdot(y_1-y_0) \\
		&=& \G1 \cdot (0.4 + 0.4\go^{-1} - 0.4) = 0.4,\\[.3cm]
		y''(t_0) &\simeq& \G1^2 \cdot F_{\mbox{\tiny{\ding{172}}}^{-1}}^2[y_{0}, y_{1}, y_2] =\G1^2 \cdot (y_2-2y_1+y_0) \\
		&=& \G1^2 \cdot (0.4+0.8\go^{-1}-0.32\go^{-2}-0.32\go^{-3} - 2(0.4 + 0.4\go^{-1}) + 0.4) \\
		&=& -0.32 -0.32\go^{-1} = -0.32 + O(\G1^{-1}), \\[.3cm]
		y'''(t_0) &\simeq& \G1^3 \cdot F_{\mbox{\tiny{\ding{172}}}^{-1}}^3[y_{0}, y_{1}, y_2, y_3] = \G1^3 \cdot (y_3-3y_2+3y_1-y_0) \\
		&=&  \G1^3 \cdot (0.4+1.2\go^{-1}-0.96\go^{-2}-1.92\go^{-3}  \\
		&& -3(0.4+0.8\go^{-1}-0.32\go^{-2}-0.32\go^{-3}) + 3(0.4 + 0.4\go^{-1}) - 0.4)\\
		&=& -0.96,
	\end{eqnarray*}
	from where we can extract the exact values of $y'(t_0)$, $y''(t_0)$, and $y'''(t_0)$ as finite parts of $0.4$, $-0.32-0.32\go^{-1}$, and $-0.96$, respectively.
	
	Let us now apply strategy (b). Here, we need to perform $k-1$ iterations of the Euler method, obtaining the values $y_1$ and $y_2$, truncating them after the grosspower $-2$:
	\begin{eqnarray*}
		f(t_0,y_0) &=& 0.4,\\
		f(t_0+\go^{-1},y_{1}) &=& 0.4 -0.32\go^{-1}-0.32\go^{-2},\\
		f(t_0+2\go^{-1},y_2) &=& 0.4-0.64\go^{-1}-1.6\go^{-2}.
	\end{eqnarray*}
	Applying  formulae (\ref{S2}), we obtain
	\begin{eqnarray*}
		y'(t_0) &=& \G1^0 \cdot F_{\mbox{\tiny{\ding{172}}}^{-1}}^0[f(t_0,y_{0})] = f(t_0, y_0) = 0.4,\\[.3cm]
		y''(t_0) &\approx& \G1^1 \cdot F_{\mbox{\tiny{\ding{172}}}^{-1}}^1[f(t_0,y_{0}), f(t_0+\go^{-1},y_1)]\\
		&=& \G1 \cdot (f(t_0+\G1^{-1}, y_1)-f(t_0,y_0))\\ 
		&=& \G1 \cdot (0.4 -0.32\go^{-1}-0.32\go^{-2} - 0.4) = -0.32 - 0.32\go^{-1} \\
		&=& -0.32 + O(\G1^{-1}),
	\end{eqnarray*}
	\begin{eqnarray*}
		y'''(t_0) &\approx& \G1^2 \cdot F_{\mbox{\tiny{\ding{172}}}^{-1}}^2[f(t_0,y_{0}), f(t_0+\go^{-1},y_1), f(t_0+2\go^{-1}, y_2)]\\
		&=& \G1^2 \cdot (f(t_0+2\G1^{-1},y_2)-2f(t_0+\G1^{-1},y_1)+f(t_0,y_0)) \\
		&=& \G1^2 \cdot (0.4-0.64\go^{-1}-1.6\go^{-2} - 2(0.4 -0.32\go^{-1}-0.32\go^{-2}) + 0.4) \\
		&=& -0.96,
	\end{eqnarray*}
	from where again we can extract the exact values of $y'(t_0)$, $y''(t_0)$, and $y'''(t_0)$ as finite parts of $0.4$, $-0.32-0.32\go^{-1}$, and $-0.96$, respectively.
	
	It should be noticed that the value $y_2$ cannot be truncated after the grosspower $-2$ using the first strategy, because the coefficient of $\G1^{-3}$ at the value $y_2$ is used also for computing $y'''(t_0)$. On the contrary, strategy (b) allows us to use the grosspowers up to $-2$, which decreases the computational cost of the procedures computing the $2-$nd and the $3-$rd derivatives.
\end{example}

%or in a similar way, we could use central differences.
%According to this latter approach, approximating the  first derivative with central differences,  the fourth-order formula (\ref{EM4}) becomes
%\begin{equation}
%\label{EM4_INF}
%\begin{array}{rl}
%y_1&=\displaystyle  y_0+\frac{h}{2}\left( f(y_0)+f(y_1) \right) + \\ &\displaystyle +\frac{h^2}{24}\left(  \frac{(f(y_1+\go^{-1}f_1)-f(y_1-\go^{-1} f_1)}{\go^{-1}} -  \frac{(f(y_0+\go^{-1}f_0)-f(y_0-\go^{-1} f_0)}{\go^{-1}}  \right).
%\end{array}
%\end{equation}

%Observe that (\ref{EM4_INF}) define a new numerical scheme with finite and infinitesimal steps.

\section{Numerical illustrations}
\label{sec:num}
In the present section,  the Euler--Maclaurin formulae of order four \eqref{EM4} and six \eqref{EM6} are applied to a few well-known test problems to highlight their conservation properties. In particular, the long-time behavior of their numerical solutions is compared with that of the numerical solutions computed by the (symplectic) Gauss methods of order four and six. 

To  advance the solution at each integrations step, the nonlinear equations (\ref{EM4}) and (\ref{EM6}) have been solved by means of a modified Newton method, using the same Jacobian of the trapezoidal scheme, which is appropriate since the neglected terms are $O(h^2)$. In order to  preserve the  conservation properties, the nonlinear scheme must be  iterated to attain the highest possible accuracy in double precision. Moreover, the derivatives have been computed using strategy (b) defined in the previous section.

The numerical experiments have been performed using Matlab  R2017b.  Though Gauss methods generally exhibit a better accuracy for a given stepsize and order, we stress that a fair comparison of the actual performance of the two classes of methods cannot take aside the computational complexity associated with their implementation and, in particular, the effort in solving the underlying nonlinear systems at each step of the integration procedure. In this respect, we notice that while the dimension of the nonlinear systems associated with the Gauss methods is proportional to the number of stages and hence to the considered order, this is not the case for the Euler--Maclaurin formulae, for which the dimension remains the same, i.e. that of the underlying continuous problem, independently of the order.  This study is a delicate issue due to the non-homogeneous platforms the two codes have been run. At the moment, the  only available emulator of the Infinity Computer Arithmetic is a c++ prototype callable in Matlab through a suitable interface which prohibits us to have an efficient implementation and therefore to execute a fair comparisons with other techniques. We stress  that the c++ emulator has been only used for the computation of the derivatives, while all the other operations have been performed using the standard double precision floating point arithmetic available in Matlab.   We checked the obtained results with those provided by computing the derivatives analytically. %and by using   ADiGator,  a MATLAB Automatic Differentiation Tool \cite{ADIGATOR}. 

{ With this premise,  to gain some preliminary insight into the potentialities of multi-derivative methods, we have prepared a couple of experiments involving the pendulum and the Kepler problems, for which the periodic nature of the solution allows us to accurately estimate the error. In particular, taking into account the discussion above, we compare the fourth-order Euler-Maclaurin and Gauss methods on the basis of their computational complexity but under the assumption that the derivative $f'(y)f(y)$ needed by the the former integrator be analytically evaluated. The results have been discussed in the next two subsections and collected in Figure \ref{wpd_fig}.}

	%\marginpar{\tiny \color{blue} Completare con qualche altro dettaglio tecnico?}

	%This to stress that it is possible to implement mixed codes that uses  the Infinity Computation only when it is required. 

\subsection{Nonlinear pendulum}
\label{pend}

\begin{figure}
	\centerline{
		\includegraphics[width=8.0cm,height=4cm]{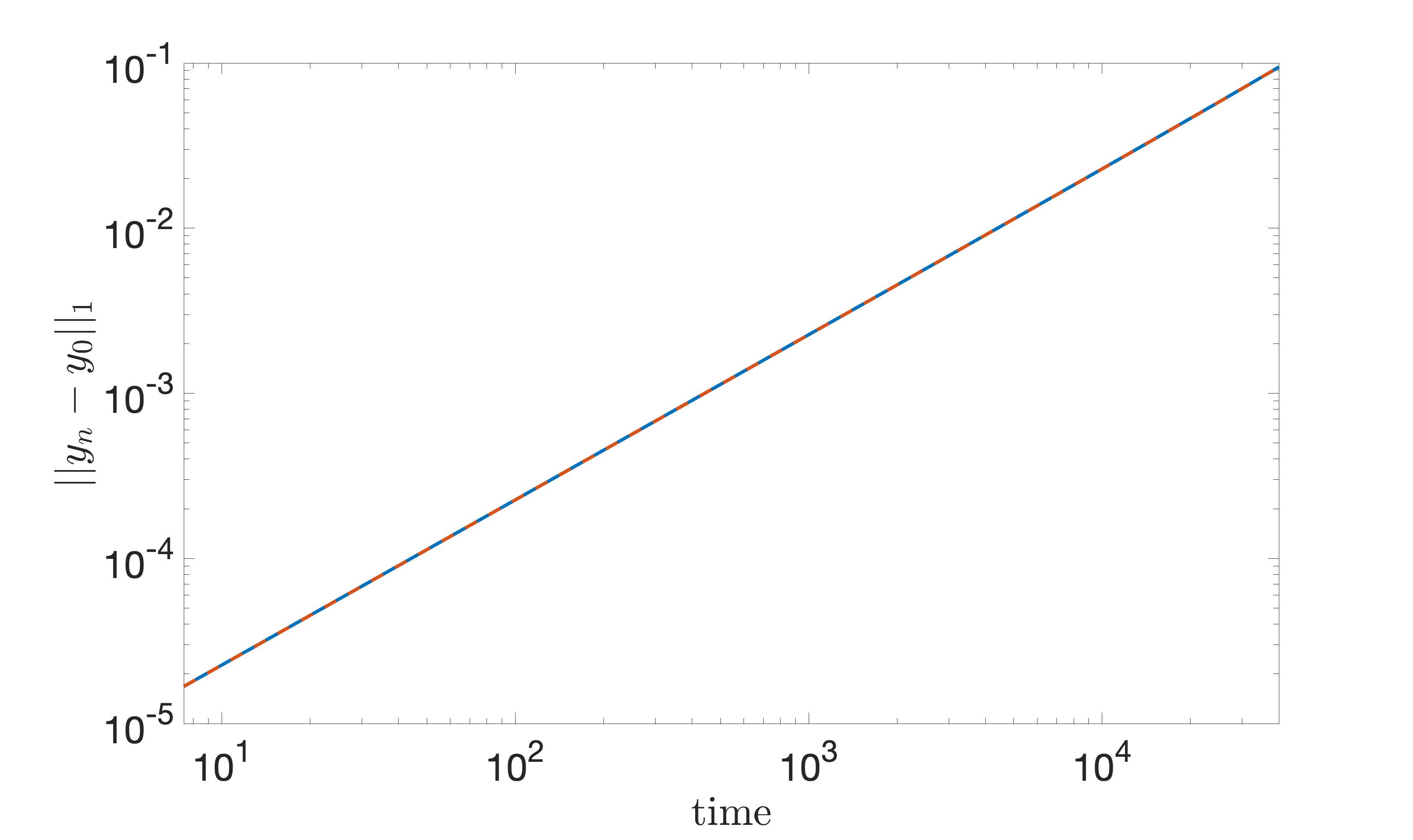} \hspace*{-.6cm}
		\includegraphics[width=8.0cm,height=4cm]{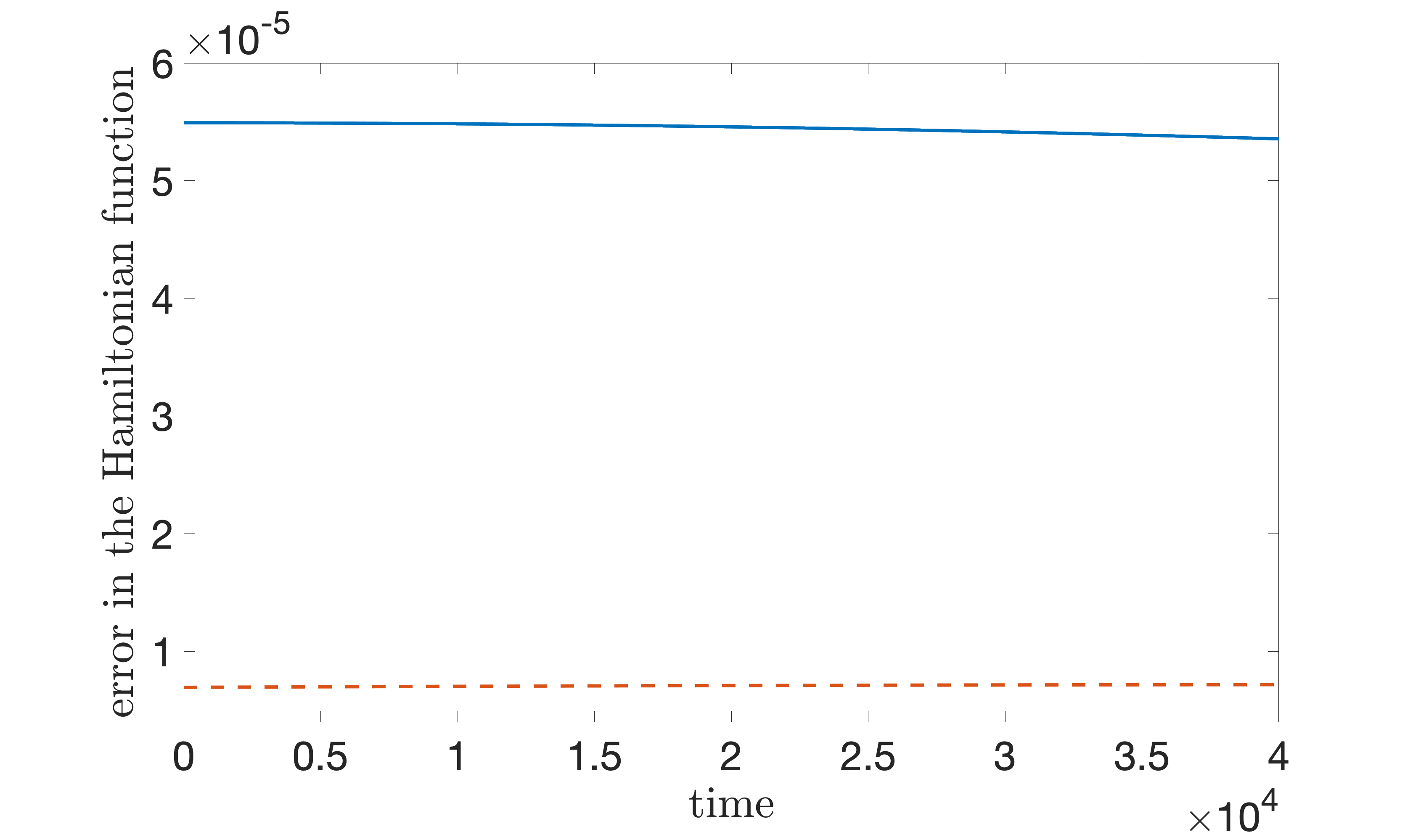}}
	\caption{Results for the fourth-order Euler-Maclaurin method (solid lines) and Gauss method (dashed lines) applied to the pendulum problem.}
	\label{pend4_fig}
\end{figure}

As a first example, we consider the dynamics of a pendulum under influence of gravity. It is usually described in terms of the angle $q$ that the pendulum forms with its  stable rest position:
\begin{equation}\label{pendulum}
\ddot q + \sin q = 0,
\end{equation}
where $p=\dot q$ is the angular velocity. The Hamiltonian function associated with  (\ref{pendulum}) is
\begin{equation}
\label{pendulumH}
H(q,p) = \frac{1}2 p^2 -\cos q.
\end{equation}
An initial condition $(q_0,p_0)$ such that $|H(q_0,p_0)|<1$ gives rise to a periodic solution $y(t)= (q(t),p(t))^\top$ corresponding to oscillations of the pendulum around the straight-down stationary position. In particular, starting at $y_0=(q_0,0)^\top$, the period of oscillation may be expressed in terms of the complete elliptical integral of the first kind as
$$
T(q_0) = \int_0^1\frac{\mathrm{d}z}{\sqrt{(1-z^2)(1-\sin^2(q_0/2)z^2)}}.
$$
We choose $q_0=\pi/2$ to which there corresponds a period $T=7.416298709205487$. We use the fourth and sixth order Euler--Maclaurin  and Gauss  methods with stepsize $h=T/28$ to integrate the problem over $5 \cdot10^3$ periods for the fourth-order  methods and  $4\cdot10^5$ periods for the  sixth-order methods. We then compute the errors $\|y_n-y_0\|_1$  in the solution and $\max_{1 \le j \le 28} |H(y_{n+j})-H(y_0)|$ in the energy function at times multiples of the period $T$, that is for $n=28k$, with $k=1,2,\dots$. Figures \ref{pend4_fig} and \ref{pend6_fig} report the obtained results. On the left plot, we can see that the error in the solution as time increases is essentially the same for  the fourth-order formulae and quite similar for the sixth-order formulae. A near conservation of the energy function is observable on the right of each figure. %{\color{red} The amplitudes of the bounded oscillations are similar for both methods thus confirming the good long-time behavior properties of Euler--Maclaurin formulae for the problem at hand.}

\begin{figure}
	\centerline{
		\includegraphics[width=8.0cm,height=4cm]{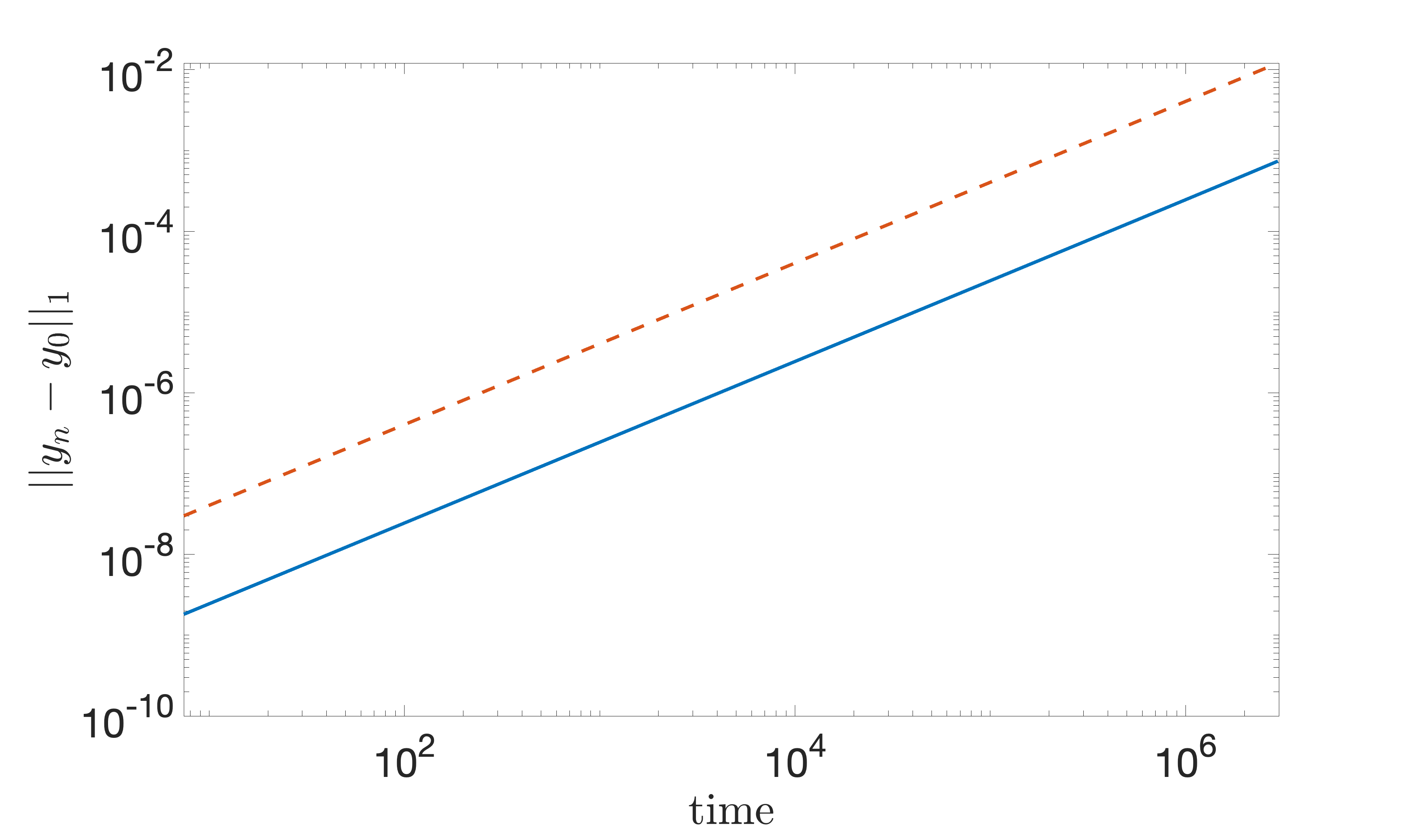} \hspace*{-.6cm}
		\includegraphics[width=8.0cm,height=4cm]{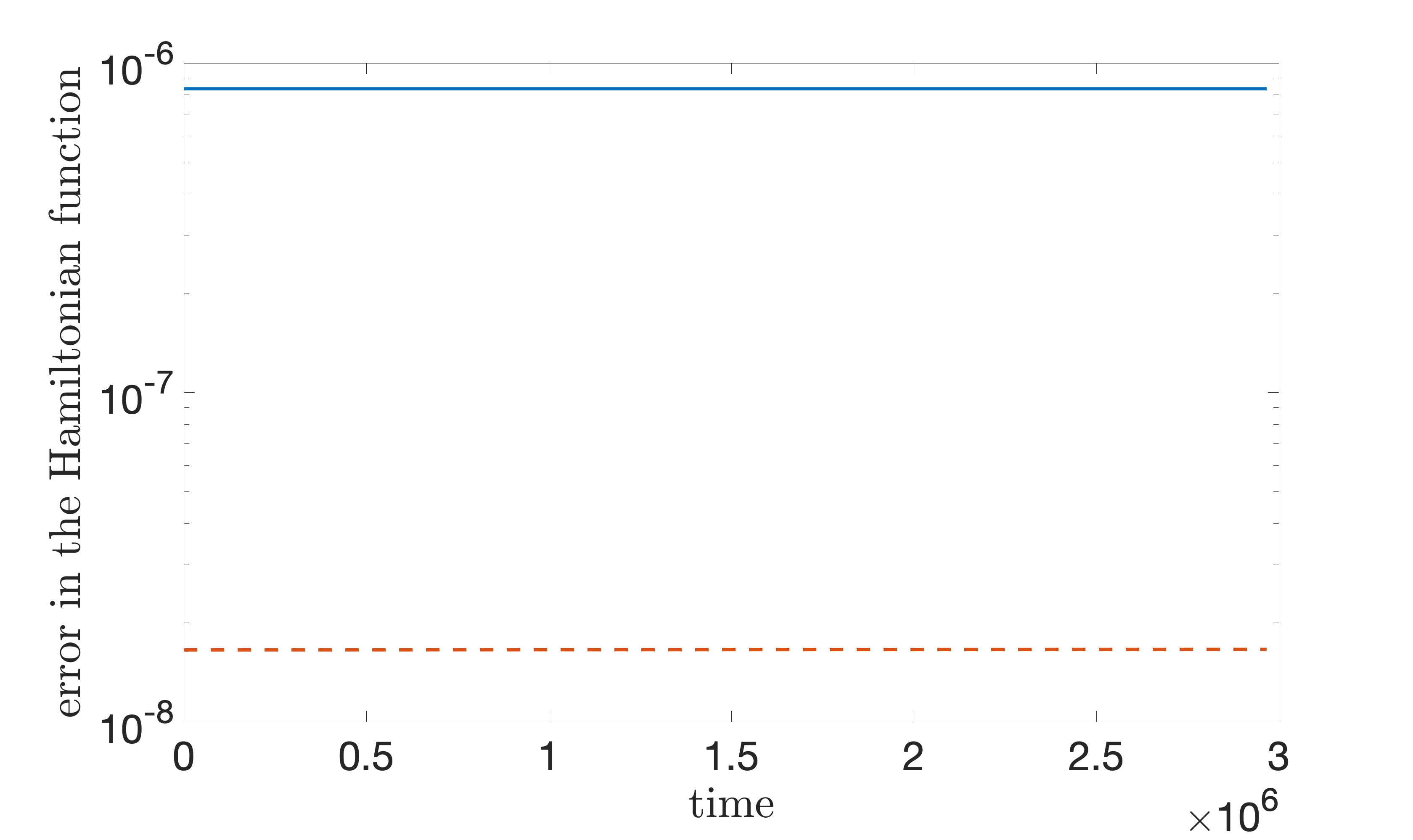}}
	\caption{Results for the sixth-order Euler-Maclaurin method (solid lines) and Gauss method (dashed lines) applied to the pendulum problem.}
	\label{pend6_fig}
\end{figure}

{The work precision diagram in the left picture of  Figure \ref{wpd_fig} plots the execution times versus the  accuracy in the numerical solutions obtained by the fourth-order Euler-Maclaurin and Gauss methods applied to integrate the pendulum problem on the interval $[0,10^2 T]$. The results show that the Euler-Maclaurin formula is indeed competitive. 
%As was anticipated in the introductory part of this section, here we have evaluated the derivative $f'(y)f(y)$ needed by the multi-derivative integrator by using its analytical expression rather than exploiting the Matlab interface of the c++ emulator of the Infinity Computer which needs a properly optimized
We emphasize that  the Gauss method has been implemented by using the efficient techniques described in \cite{BFCI14}.}
\begin{figure}
	\centerline{
		\includegraphics[width=8.0cm,height=4cm]{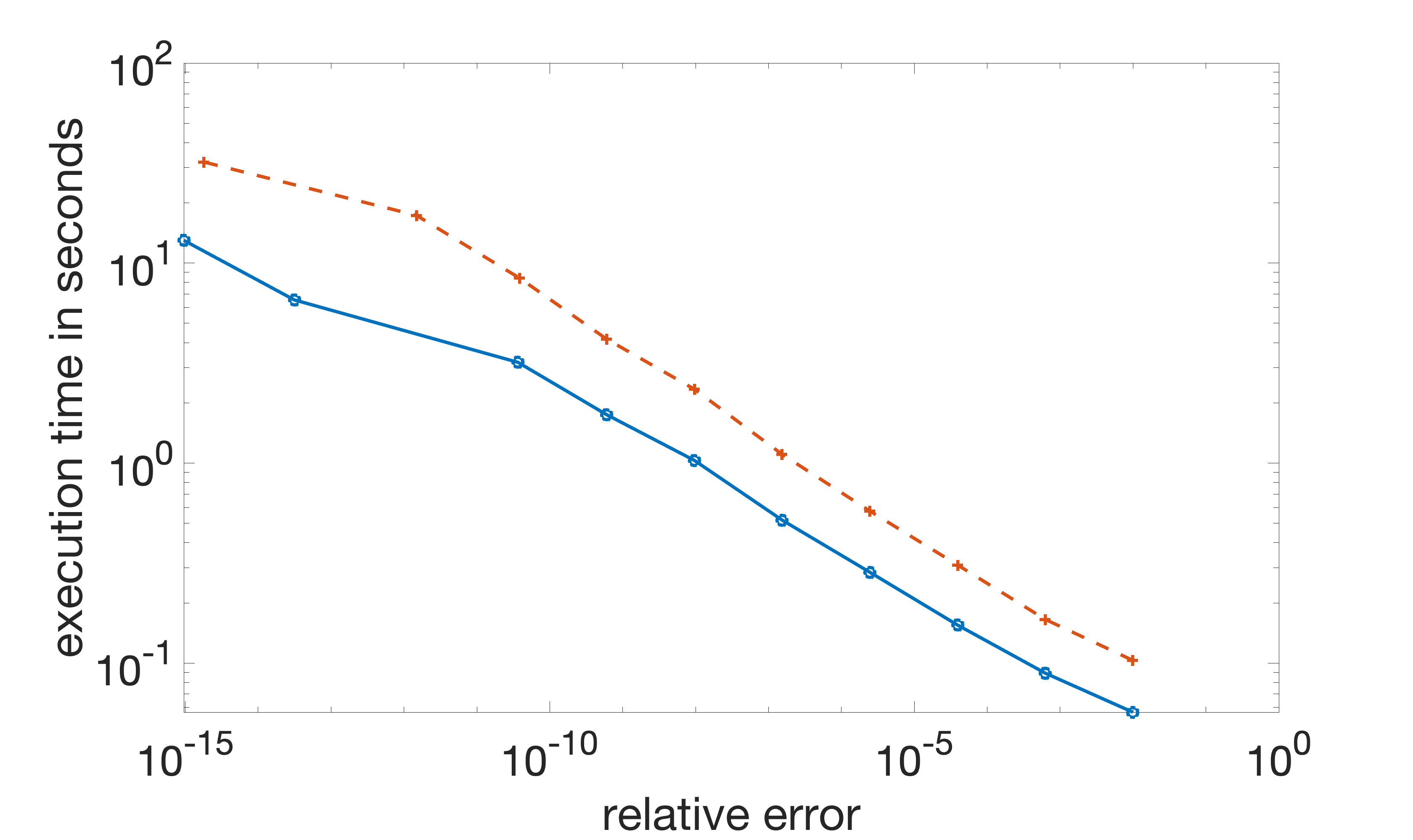} \hspace*{-.6cm}
		\includegraphics[width=8.0cm,height=4cm]{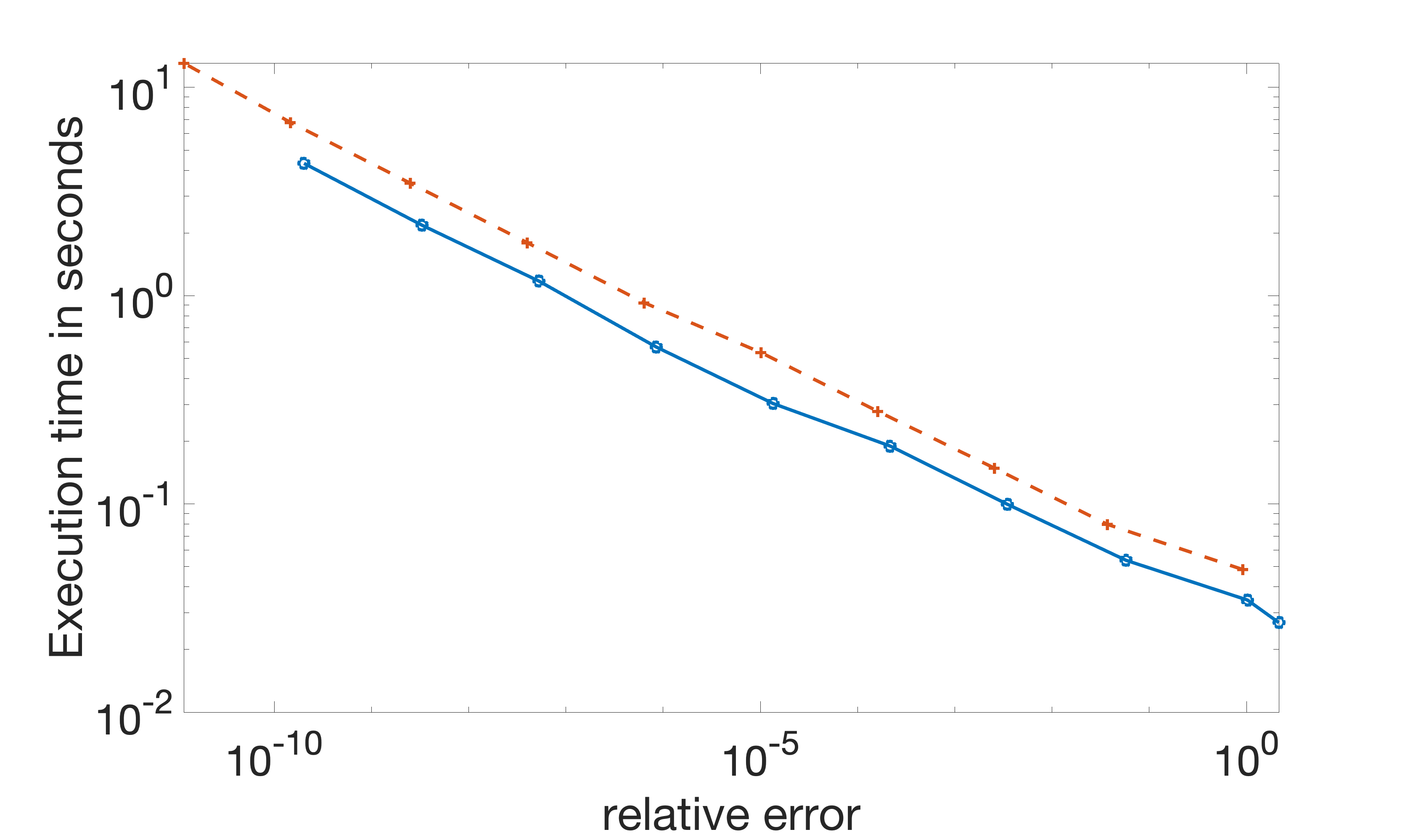}}
	\caption{Work precision diagrams for the fourth-order Euler-Maclaurin method (solid lines) and Gauss method (dashed lines)  applied to the pendulum problem (left picture) and to the Kepler problem (right picture).}
	\label{wpd_fig}
\end{figure}

\subsection{The Kepler problem}
\label{kepl}
This classical problem describes the motion of two  bodies subject to Newton's law of gravitation. As is well-known, the problem is a completely integrable Hamiltonian dynamical system with two degrees of freedom (see, for example, \cite{BI16}). If the origin of the coordinate system is set on one of the two bodies, the Hamiltonian function
$$
H(q_1,q_2,p_1,p_2) = \frac{1}{2}(p_1^2+p_2^2)-\frac{1}{\sqrt{q_1^2+q_2^2}},
$$
describes the motion of the other body, namely  an ellipse in the $q_1$-$q_2$ plane.
Taking as initial conditions
$$
q_1(0)=1-e, \quad q_2(0)=0, \quad p_1(0)=0, \quad p_2(0)=\sqrt{\frac{1+e}{1-e}},
$$
the trajectory describes an ellipse with eccentricity $e$ and is periodic with period $T=2\pi$.
Besides the total energy $H$, further  relevant first integrals  are the angular momentum
$$
M(q_1, q_2, p_1, p_2) = q_1p_2-q_2p_1.
$$
and the Lenz vector $A=(A_1,A_2,A_3)^\top$, whose components are
$$
A_1(q,p)=p_2M(q,p)-\frac{q_1}{||q||_2}, \quad A_2(q,p)=-p_1M(q,p)-\frac{q_2}{||q||_2}, \quad  A_3(q,p)=0.
$$
Of the four first integrals $H, M, A_1$, and $A_2$ only three are independent. Having set $e=0.6$ and $h=T/400$, we integrate the problem  
over $8 \cdot10^2$ periods for the fourth-order  methods and  $10^5$ periods for the  sixth-order methods and compute the error $\|y_n-y_0\|_1$   in the solution at specific times multiples of the period $T$, that is for $n=400k$, with $k=1,2,\dots$. Figures \ref{kepler4_fig} and \ref{kepler6_fig} report the obtained results for the fourth and sixth order Euler--Maclaurin (solid lines) and Gauss (dashed lines) methods. On the top-left picture is the absolute error of the numerical solution; the top-right picture shows the  maximum error  in the Hamiltonian function in each period; the error in the angular momentum, also evaluated in each period, is drawn in the bottom-left picture while the bottom-right picture concerns the  maximum error in each period of the Lenz vector. As is expected, we can see a linear drift in the error $\|y_n-y_0\|_1$  as the time increases.  The same linear growth is experienced in the Lenz invariant. Euler--Maclaurin methods assure a near conservation of the Hamiltonian function and angular momentum. This latter quadratic invariant is precisely conserved (up to machine precision) by Gauss methods due to their symplecticity property.

\begin{table}
	\centerline{
		\begin{tabular}{|c|cc|cc|}
			\hline
			&	\multicolumn{2}{c|}{order 4}  & 	\multicolumn{2}{c|}{order 6}  \\
			N &  error &  rate & error & rate \\
			\hline
			32    & 8.47e-03  & 4.10 & 2.59e-03  & 6.40\\
			64   &4.92e-04  & 4.01  & 3.07e-05  & 6.08\\
			128  &3.04e-05   &4.00  & 4.53e-07   &5.99\\
			256 &1.90e-06   &4.00  &7.10e-09   &6.00\\
			512   &1.18e-07   &4.00   &1.11e-10   &5.99\\  
			1024   &7.42e-09  & 4.00   &1.73e-12  & 5.77\\
			% 2048    &  4.63e-10   &3.99  & 3.16e-14   &2.54\\
			% &2.90e-11   &3.97   &5.41e-15  &  *\\
			%  & 1.84e-12   &4.03   &1.14e-14   & *\\
			\hline
	\end{tabular}}
	\caption{Convergence rate of the error in the angular moment for the  fourth-order and sixth-order  Euler-Maclaurin methods. $N$ is the number of mesh points in each period, the error is computed over 10 periods.}
	\label{am_table}
\end{table}

\begin{figure}
	\centerline{
		\includegraphics[width=8.0cm,height=4cm]{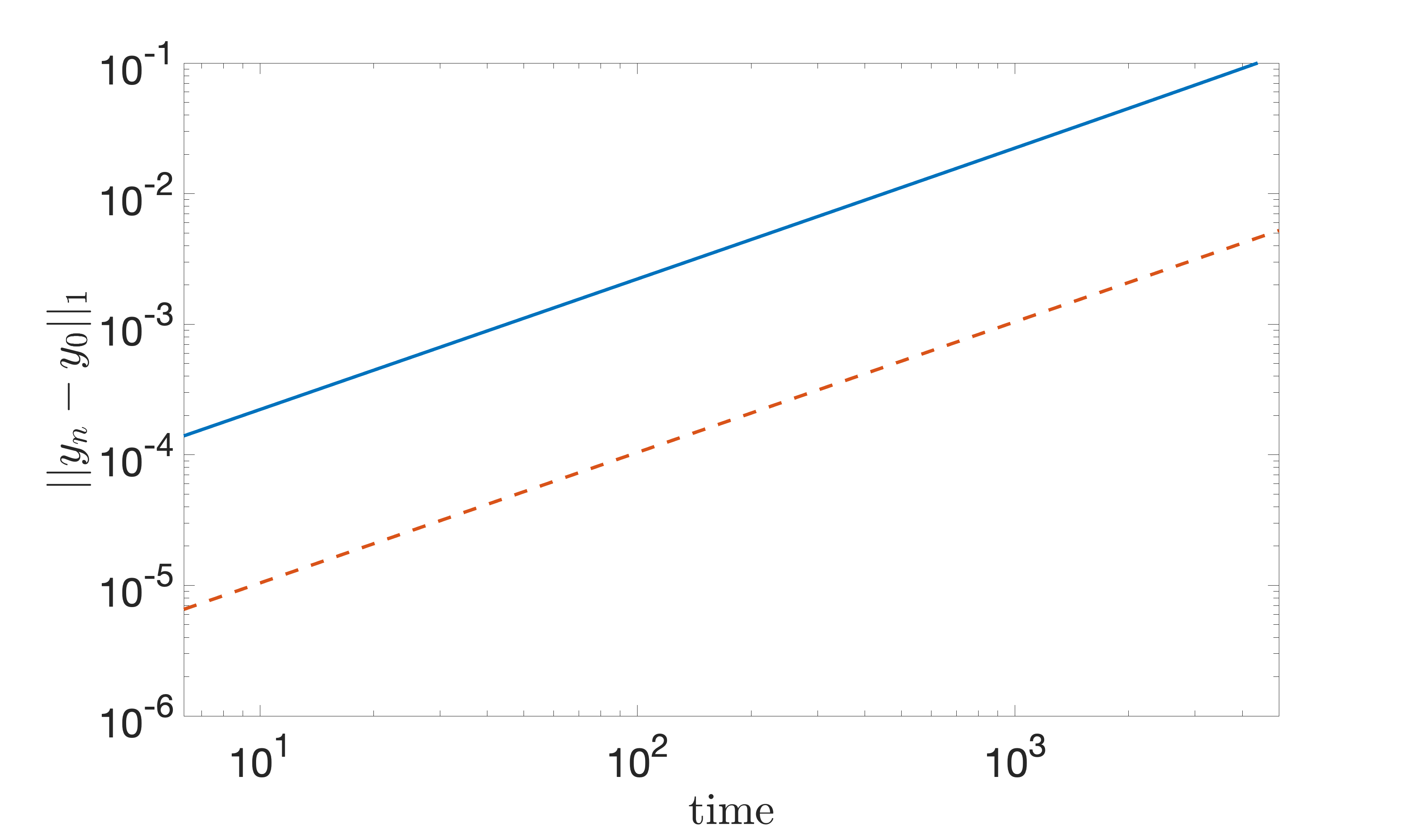} \hspace*{-.6cm}
		\includegraphics[width=8.0cm,height=4cm]{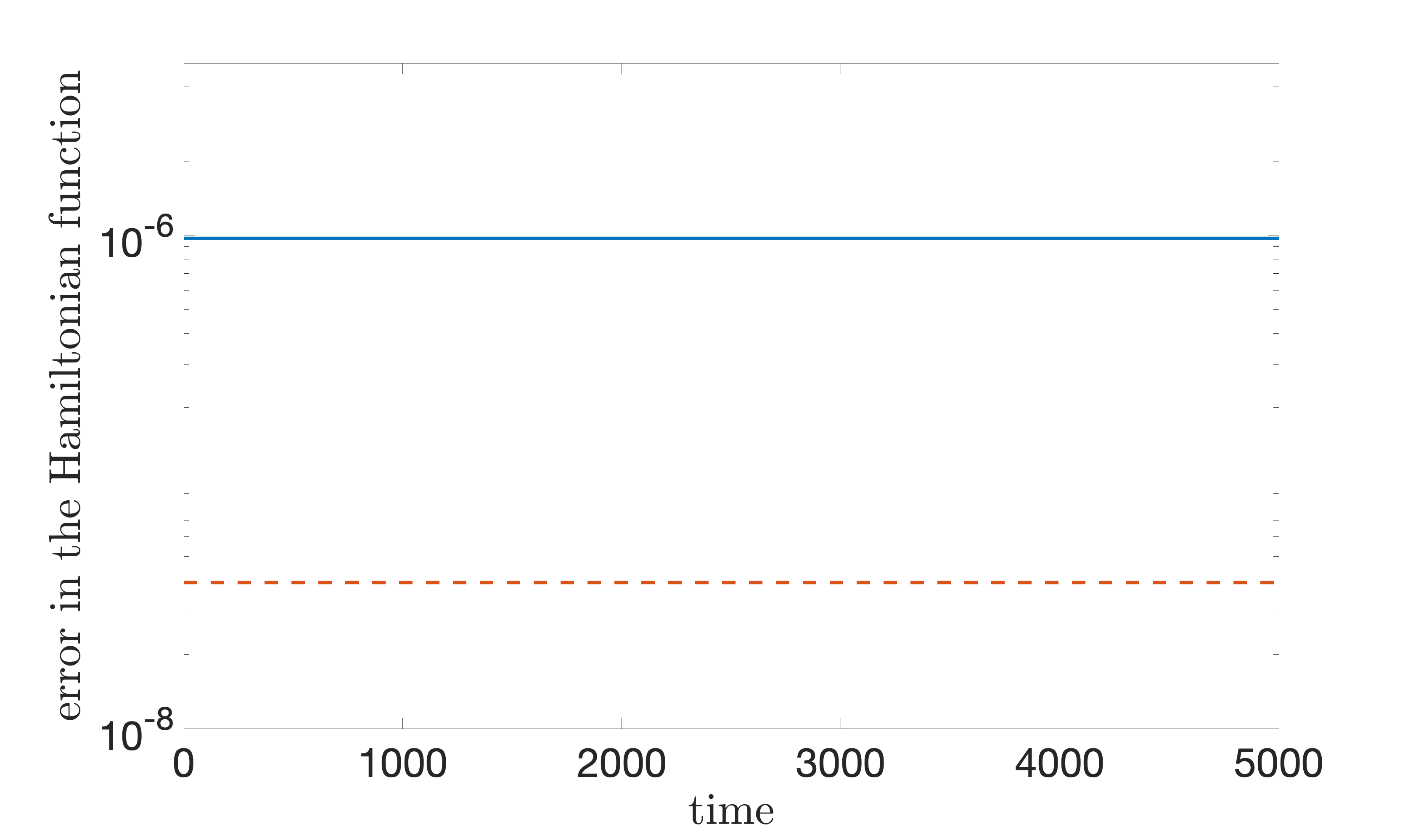}}
	\centerline{
		\includegraphics[width=8.0cm,height=4cm]{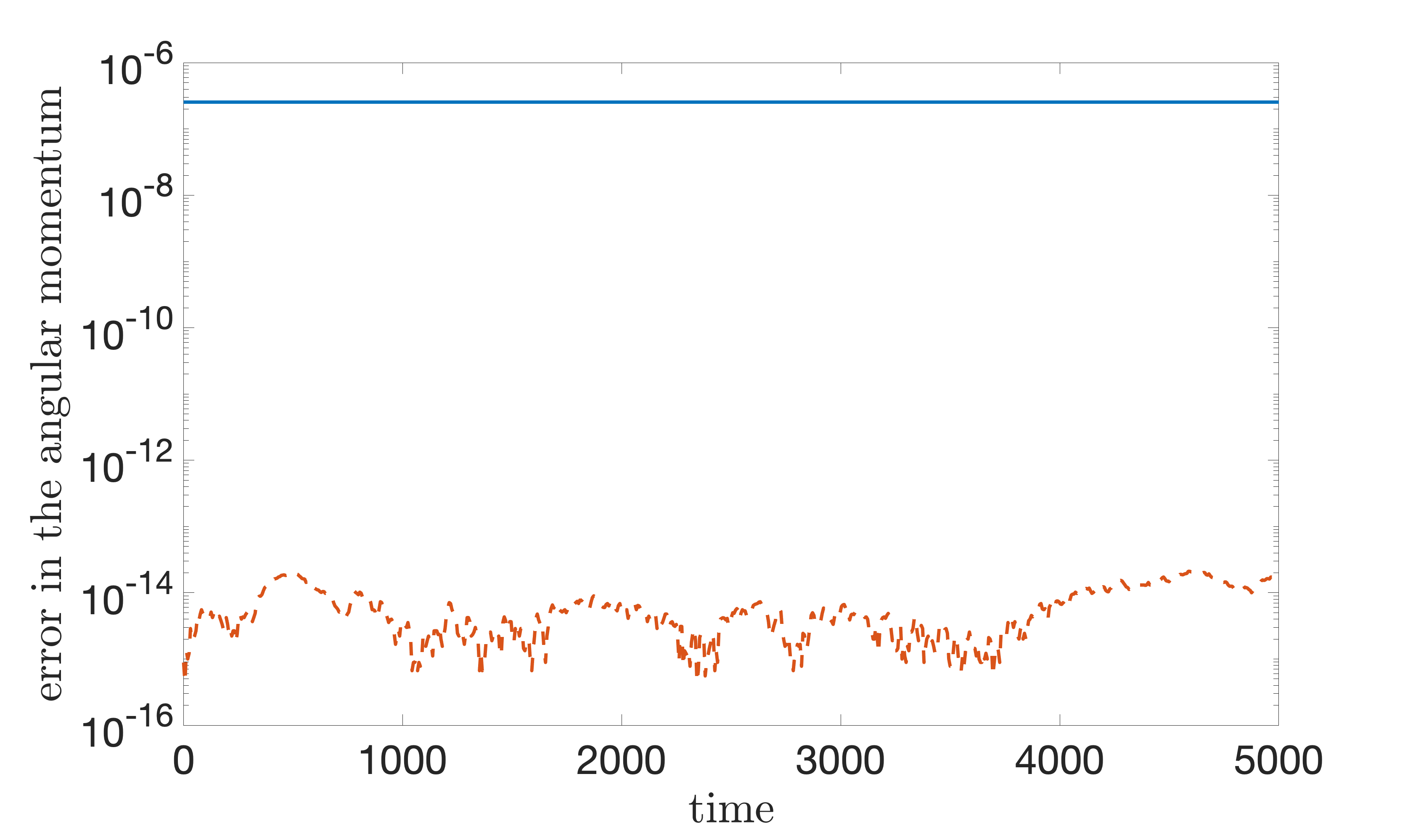} \hspace*{-.6cm}
		\includegraphics[width=8.0cm,height=4cm]{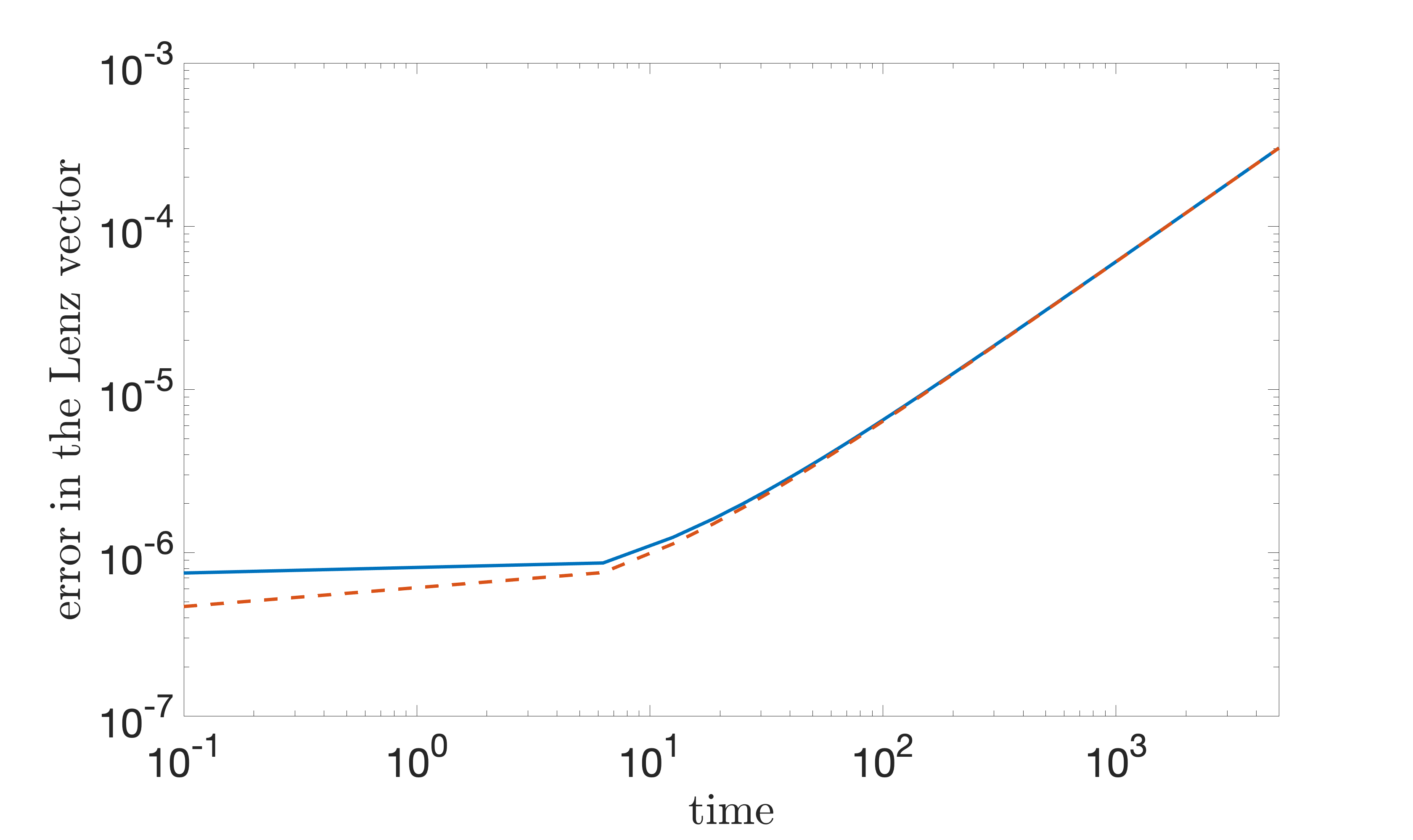}}
	\caption{Results for the fourth-order Euler-Maclaurin method (solid lines) and Gauss method (dashed lines) applied to the Kepler problem.}
	\label{kepler4_fig}
\end{figure}
\begin{figure}
	\centerline{
		\includegraphics[width=8.0cm,height=4cm]{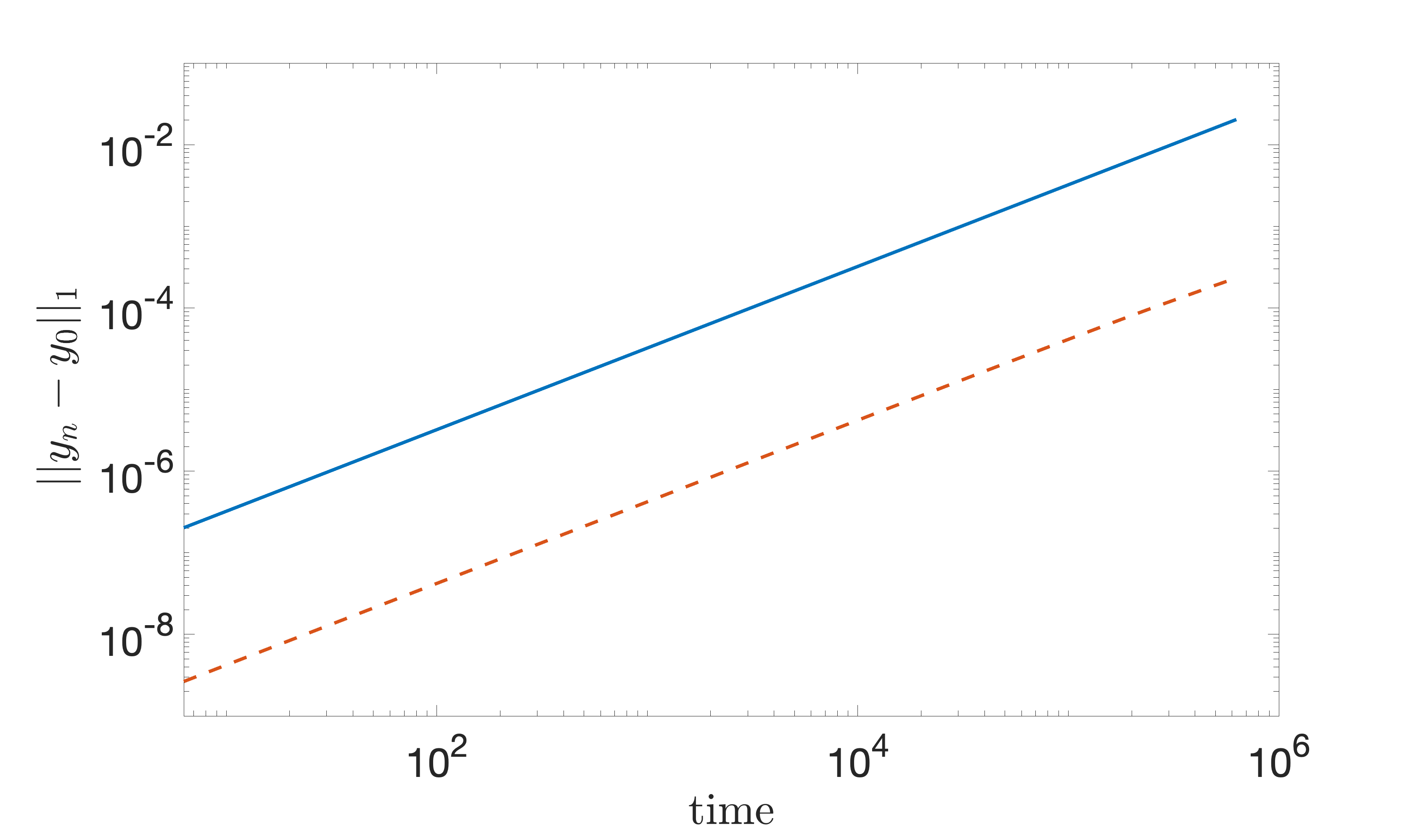} \hspace*{-.6cm}
		\includegraphics[width=8.0cm,height=4cm]{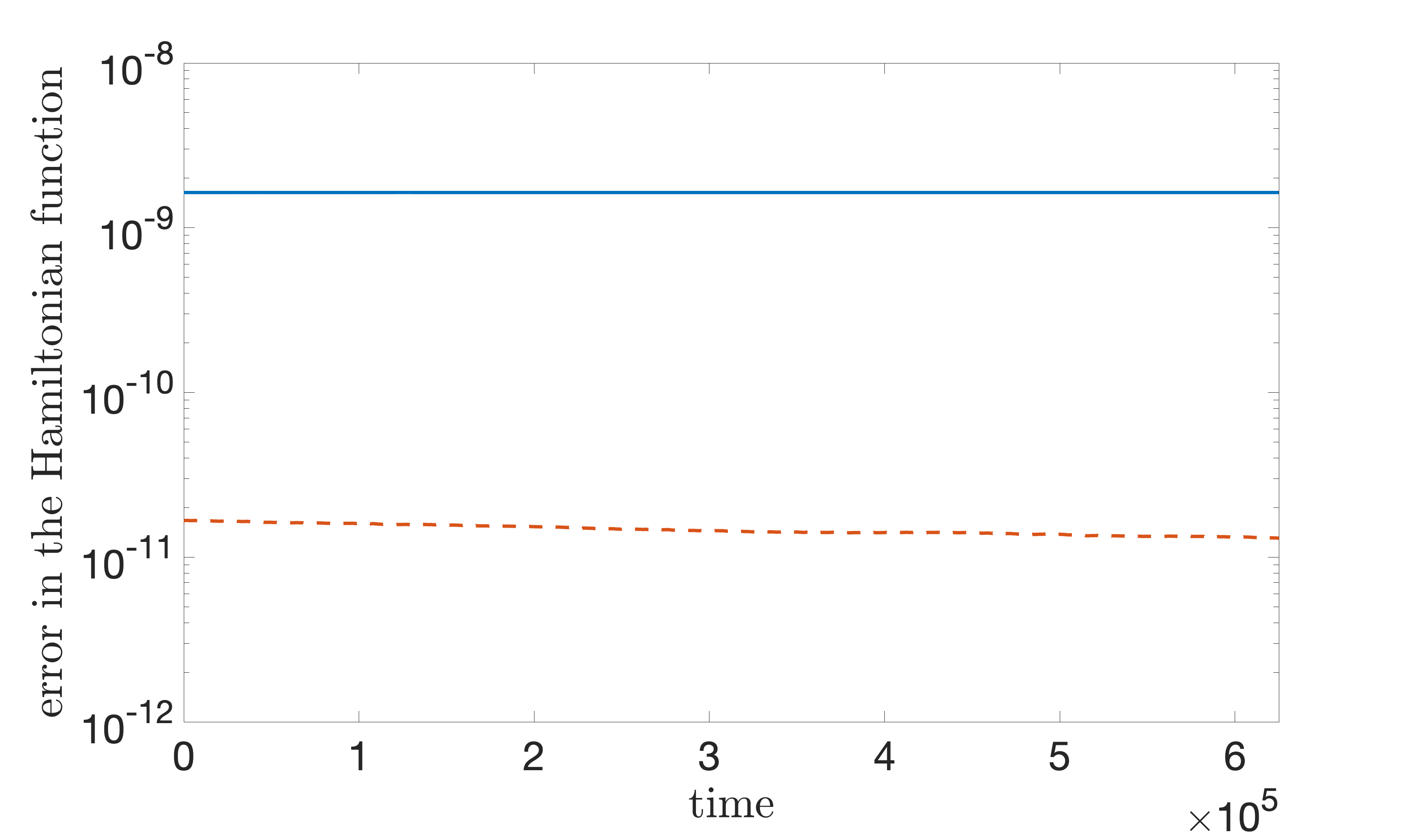}}
	\centerline{
		\includegraphics[width=8.0cm,height=4cm]{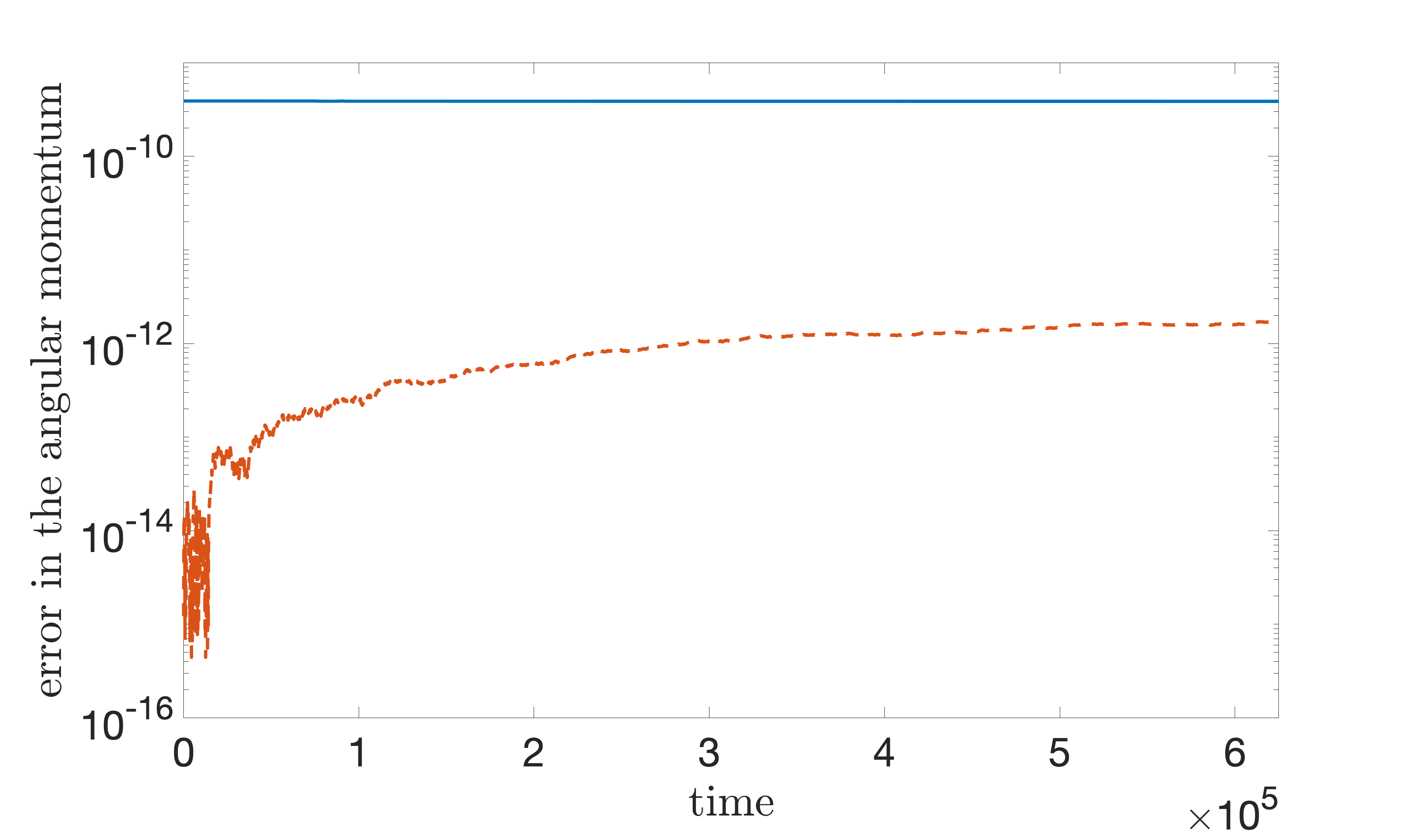} \hspace*{-.6cm}
		\includegraphics[width=8.0cm,height=4cm]{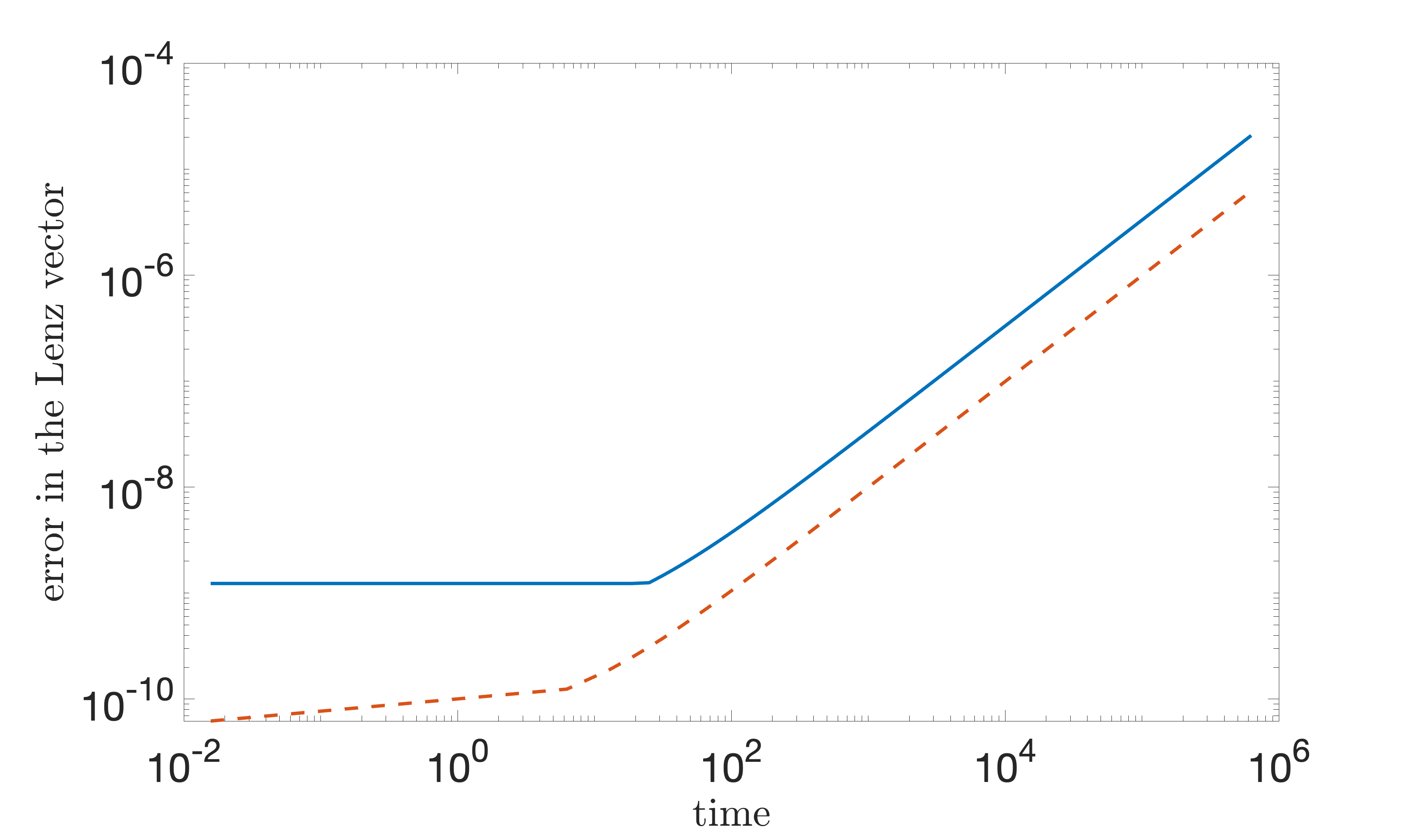}}
	\caption{Results for the sixth-order Euler-Maclaurin method (solid lines) and Gauss method (dashed lines) applied to the Kepler problem.}
	\label{kepler6_fig}
\end{figure}

\begin{figure}
	\centerline{
		\includegraphics[width=8.0cm,height=4cm]{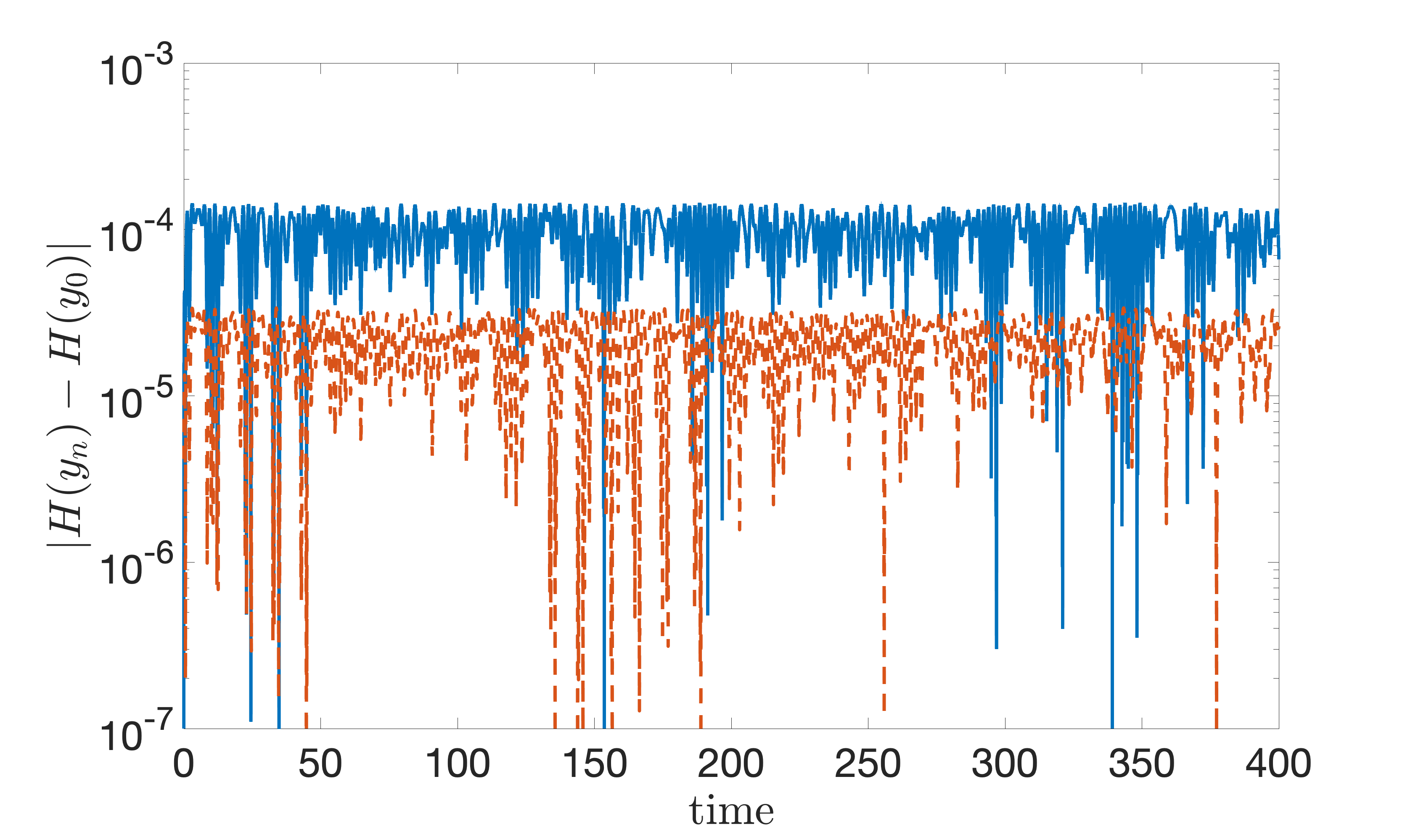} \hspace*{-.6cm}
		\includegraphics[width=8.0cm,height=4cm]{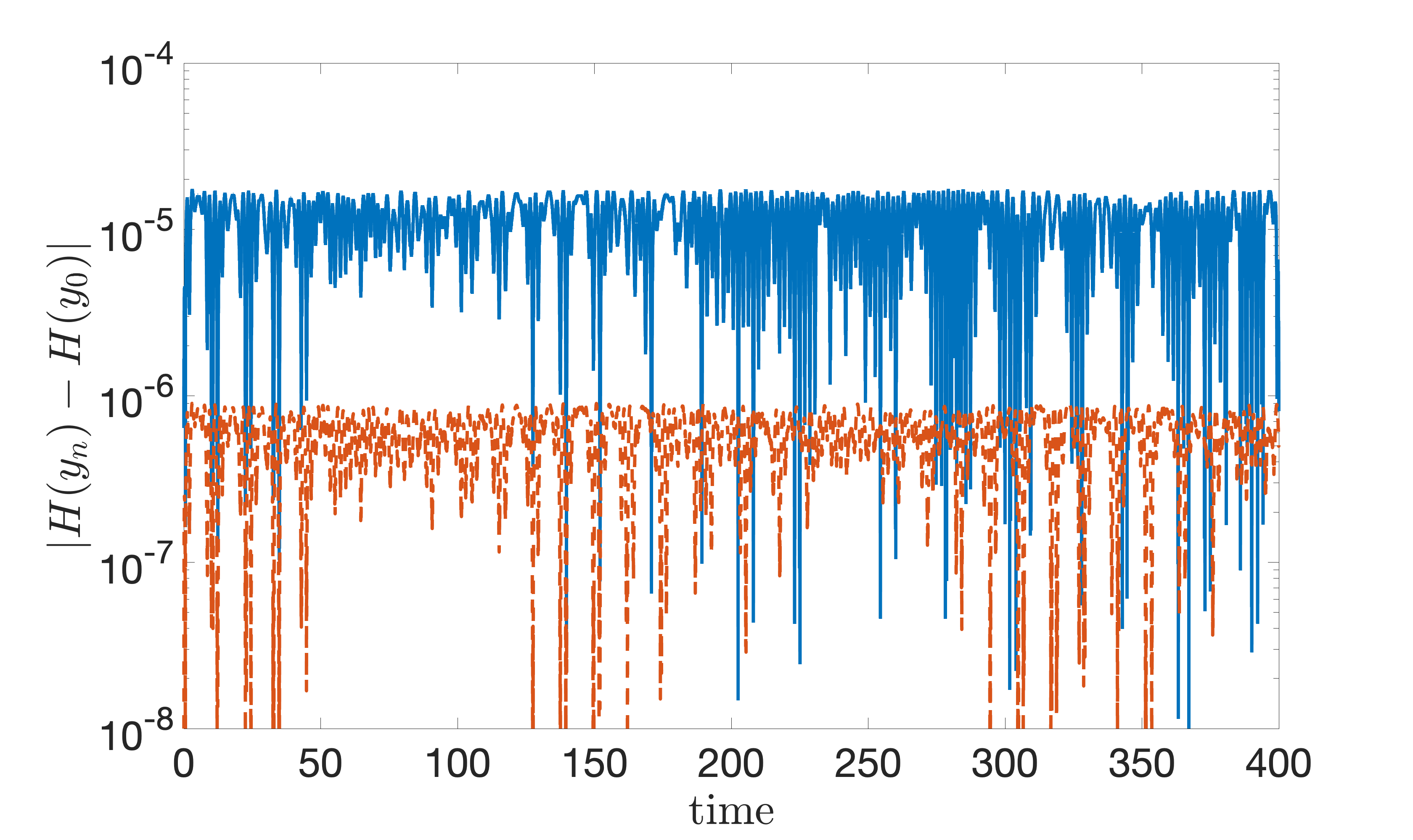}}
	\caption{Error in the Hamiltonian function \eqref{fpuH} generated by the Euler Maclaurin methods (solid line) and Gauss methods (dashed line) of order 4 (left picture) and order 6 (right picture) applied to the Fermi-Pasta-Ulam problem.}
	\label{fpuH_fig}
\end{figure}

\begin{figure}
	\centerline{
		\includegraphics[width=8.0cm,height=4cm]{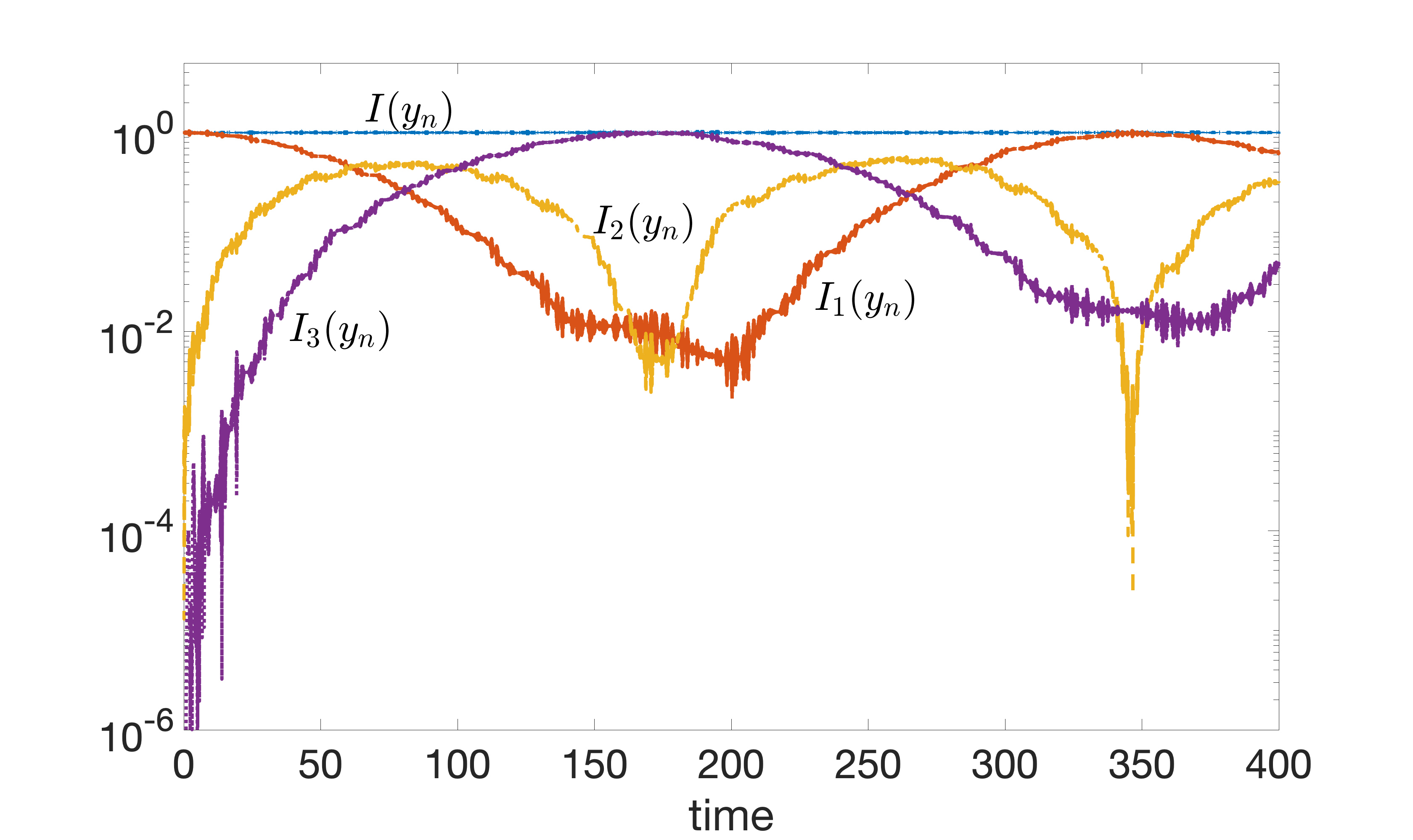} \hspace*{-.6cm}
		\includegraphics[width=8.0cm,height=4cm]{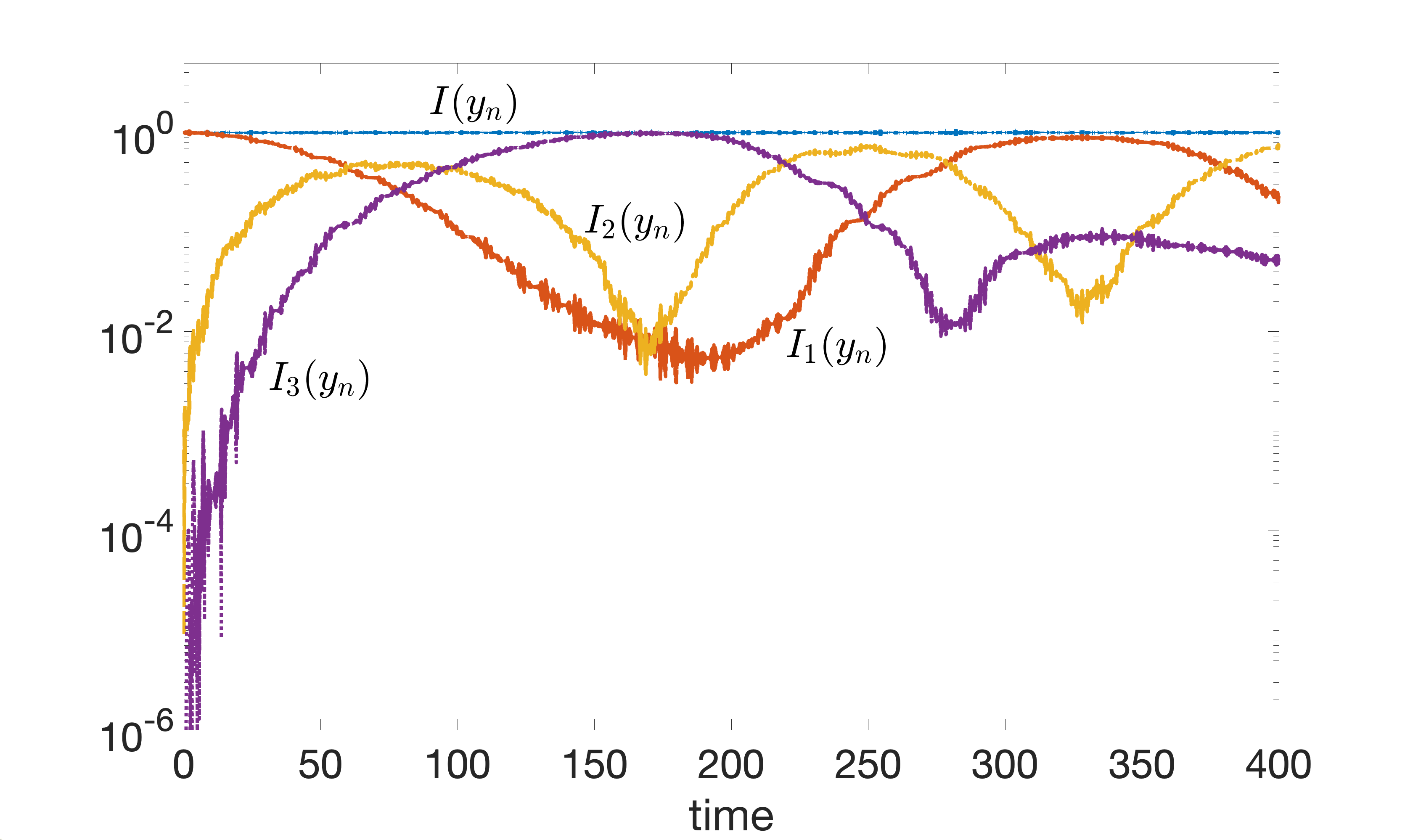}}
	\centerline{
		\includegraphics[width=8.0cm,height=4cm]{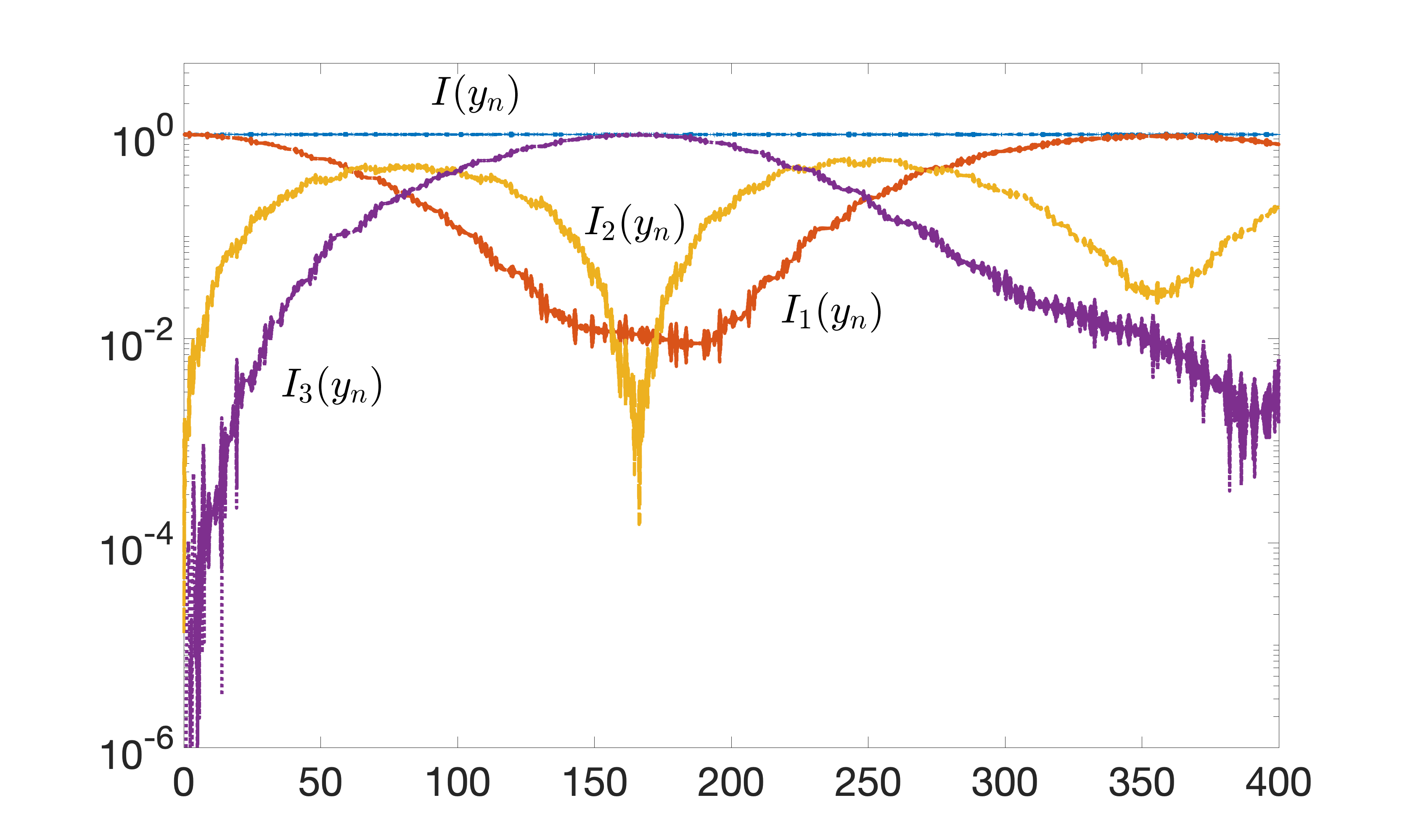} \hspace*{-.6cm}
		\includegraphics[width=8.0cm,height=4cm]{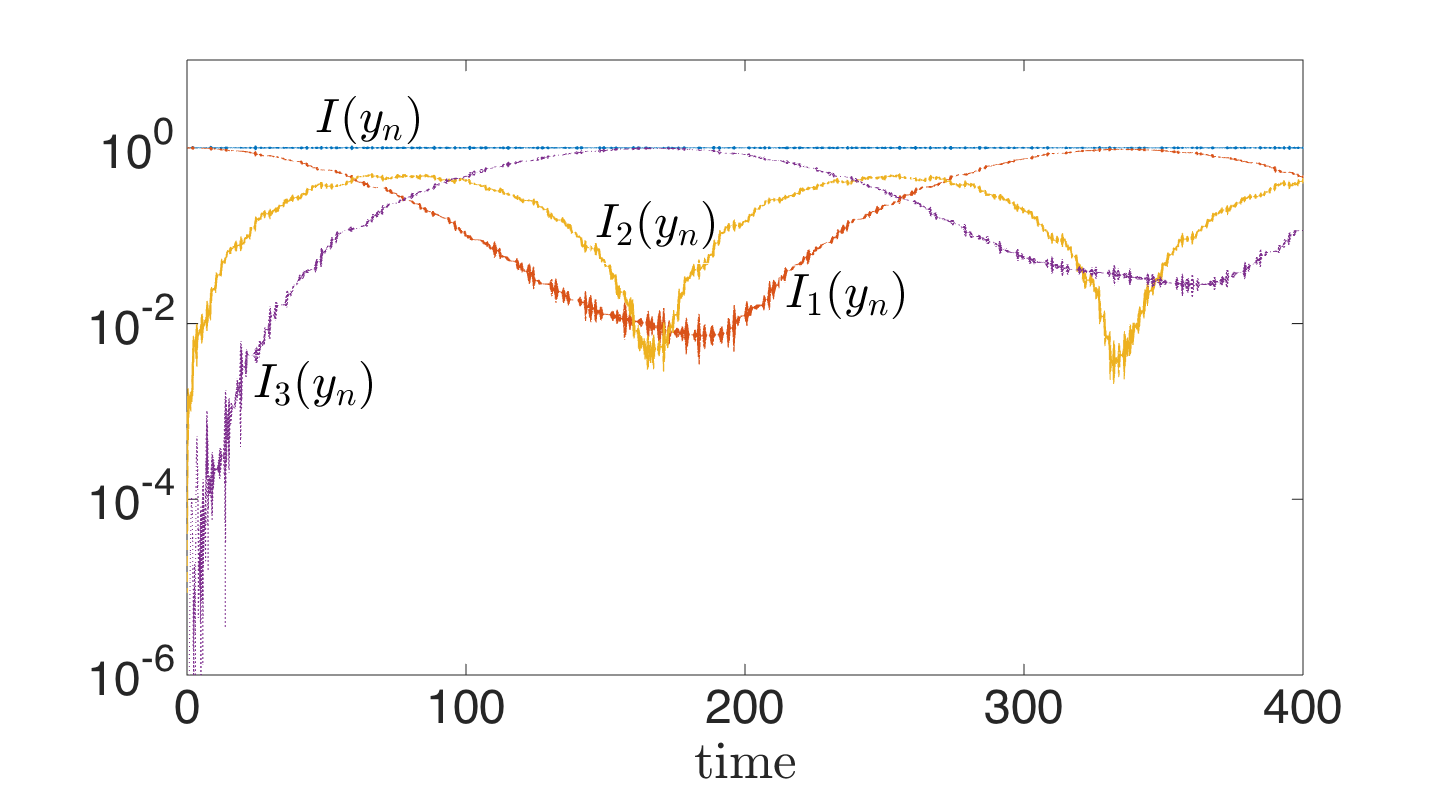}}
	\caption{Fermi-Pasta-Ulam problem: energy functions $I_i(t)$ associated with each linear spring and their sum $I(t)$ computed by Euler-Maclaurin methods (left pictures) and Gauss methods (right pictures). The upper and bottom pictures refer to the formulae of order four and six respectively.}
	\label{fpuI_fig}
\end{figure}
{Finally, also for this problem, we have run the fourth-order Euler-Maclaurin and Gauss methods for decreasing values of the stepsize over $10$ periods, and stored the corresponding errors and execution times. The work precision diagram in the right picture of  Figure \ref{wpd_fig} collects the obtained results and  shows a very similar behavior of the two integrators. The convergence rate of the maximum absolute error in the angular moment,   computed over the all integration interval,  is consistent with the order of the methods, testifiyng the near conservation of this first quadratic integral, see Table  \ref{am_table}.}

\subsection{Fermi-Pasta-Ulam problem}
\label{FPU}
The  Fermi-Pasta-Ulam problem models a physical system composed by $2m$ unit point masses disposed along a line and chained together by alternating weak nonlinear springs and stiff linear springs\cite{BI16,HLW06}.  The force exerted by the nonlinear springs are assumed proportional to the cube of the displacement of the associated masses. Denoting by $q_1$, $q_2$, $\dots$, $q_{2m}$ the displacements of the  masses from their rest points and assuming the endpoints of the external springs to be fixed, $q_0=q_{2m+1}=0$, the resulting Hamiltonian problem is defined by the energy function
\begin{equation}
\label{fpuH}
H(q,p) = \frac{1}2\sum_{i=1}^m (p_{2i-1}^2+p_{2i}^2)+\frac{\omega^2}{4}\sum_{i=1}^m(q_{2i}-q_{2i-1})^2+\sum_{i=0}^m(q_{2i+1}-q_{2i})^4,
\end{equation}
where $p_i=\dot q_i$, $i=1,\dots,2m$ are the conjugate momenta, and $\omega$ is the stiffness coefficient of the linear springs. Following the discussion in \cite[page 22]{HLW06}, we introduce the energy $I_i$ associated with the $i$th linear spring
$$
I_i= \frac{1}{4}\left( (p_{2i}-p_{2i-1})^2 +\omega^2 (q_{2i}+q_{2i-1})\right).
$$
The total energy $I(t)=I_1+\dots+I_m$ brought by the linear springs satisfies $I(t)=I(t_0)+O(\omega^{-1})$, so that it is  almost conserved for large values of the stiffness coefficient $\omega$. In our experiments we choose $m=3$ and $\omega=50$ and integrated  the problem on the interval $[0,400]$ with stepsize $h=0.03$  and the initial values $p_0 = (0, \sqrt{2}, 0, 0, 0, 0)$ and $q_0 = ( \frac{\sqrt{2}}{2}-\frac{\sqrt{2}}{2w}, \frac{\sqrt{2}}{w}+\frac{\sqrt{2}}{2}-\frac{\sqrt{2}}{2w}, 0, 0, 0, 0)$.  The two pictures in Figure \ref{fpuH_fig} display the absolute error in the Hamiltonian function  \eqref{fpuH} evaluated along the numerical solution produced by Euler--Maclaurin (solid line) and Gauss (dashed line) methods. On the left are the results for the fourth-order formulae, while on the right are the results for the sixth order formulae. We can see that, for both methods, the Hamiltonian function is nearly conserved. The pictures in Figure \ref{fpuI_fig} suggest that the very same conclusions apply to the nearly conserved quantity  $I(t)$ above defined: there is an exchange of energy among the linear modes but the total energy does not exhibit any drift.

\subsection{A non-separable Hamiltonian problem}
As last example, we illustrate the behavior of the Euler-Maclaurin formulae on the non-separable problem defined by the Hamiltonian function
\begin{equation}
\label{cassini_prob}
H(q,p)= (q^2+p^2)^2-2a^2(q^2-p^2).
\end{equation}
The level curves defining the trajectories in the phase plane are the well-known Cassini ovals (see, for example, \cite[page 53]{BI16}).  The shape of the orbit originating from a point $(q_0,p_0)$ depends on the value of the quantity $r=(1+H(q_0,p_0)/a^2)^{1/4}$. The three possible cases are summarized in the left picture of Figure \ref{cassini_1_fig}. For $r<1$, the orbit is a loop surrounding one of the foci $F_1=(-a,0)$ or $F_2=(a,0)$ and entirely lying on the semi-plane $q>0$ or $q<0$, respectively. For $r>1$, the orbit is an oval or bone-shaped loop  embracing both the foci. The {figure-eight}  level curve corresponding to $r=1$ is the lemniscate of Bernoulli and acts as a separatrix between the two possible dynamics described above. 

Consequently, a correct reproduction of the orbit when $H(q_0,p_0)\approx 0$ requires the use of an integrators with good geometric properties. To show that this is indeed the case, in the right picture of Figure \ref{cassini_1_fig} we display the orbits originating from the point $(q_0,p_0)=(0,10^{-2})$ computed by the Euler-Maclaurin formula of order 4, and by the explicit and implicit Taylor methods of the same order. This latter multi-derivative method is defined by considering a truncated Taylor expansion of $f(y(t))$ at time $t_{n+1}=t_n+h$ thus generalizing the implicit Euler method. The value of the Hamiltonian function associated with the true solution is  $H(q_0,p_0)\approx 5\cdot 10^{-4}$ corresponding to a value of $r \approx 1.00005 >1$; the integration time interval is $[0,45]$, while the used stepsize is $h=1.5\cdot10^{-2}$. Both Taylor methods exhibit a dramatic drift in the Hamiltonian function thus failing to yield the correct shape of the orbit even for short times. On the contrary, the near conservation of (\ref{cassini_prob}) along the numerical solution generated by the Euler-Maclaurin integrator guarantees a very accurate orbit determination over long simulation times. 

This feature is further confirmed by the results displayed in Figure \ref{cassini_2_fig} where we compare the qualitative behavior of the fourth-order Euler-Maclaurin and Gauss methods in reproducing the orbit originating at  $(q_0,p_0)=(0,10^{-6})$. This orbit is characterized by a period  $T\approx3.131990057003955$ and a  value of $r\approx 1+10^{-13}$ and thus is very close (but external) to the lemniscate (see \cite[page 135]{BI16}). We solved  the problem over the time interval $[0,5\cdot 10^4 T]$ by using a stepsize $h=2.5\cdot 10^{-3} T$, corresponding to $400$ mesh points in each period. The left picture of Figure  \ref{cassini_2_fig} contains the two undistinguishable numerical trajectories in the phase plane together with two close-ups  showing that, for both methods, the computed orbit is correctly bounded away from the origin. As is clear from the right picture of  Figure \ref{cassini_2_fig}, this good asymptotic behavior is a consequence of the fact that the numerical Hamiltonian function $H(q_n,p_n)$ undergoes very small and bounded oscillation which prevent the trajectory to cross the lemniscate at any observed time.

%The two pictures in Figure  (\ref{cassini_prob})  Euler-Maclaurin methods

\begin{figure}
	\centerline{
		\includegraphics[width=8.0cm,height=4cm]{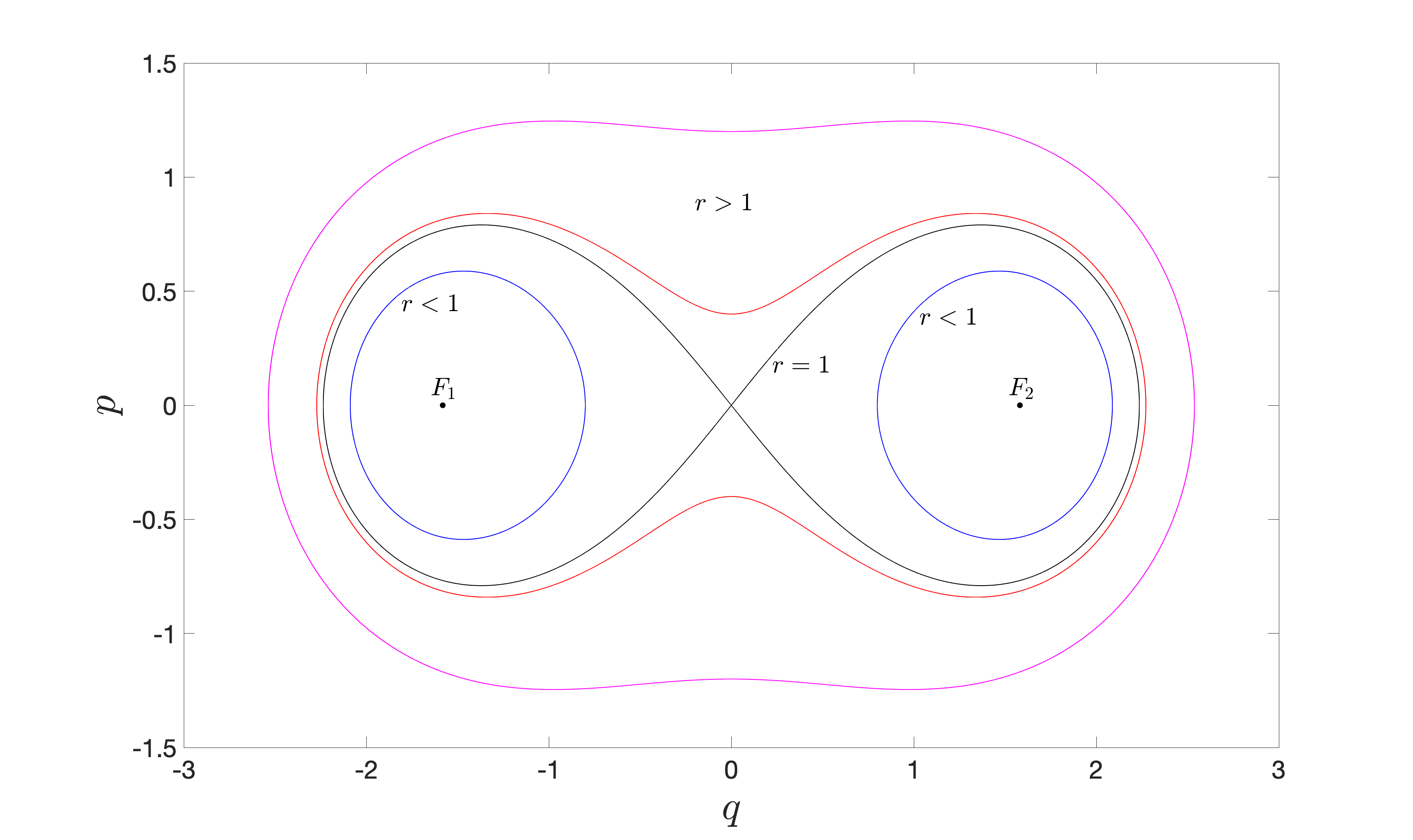} \hspace*{-.6cm}
		\includegraphics[width=8.0cm,height=4cm]{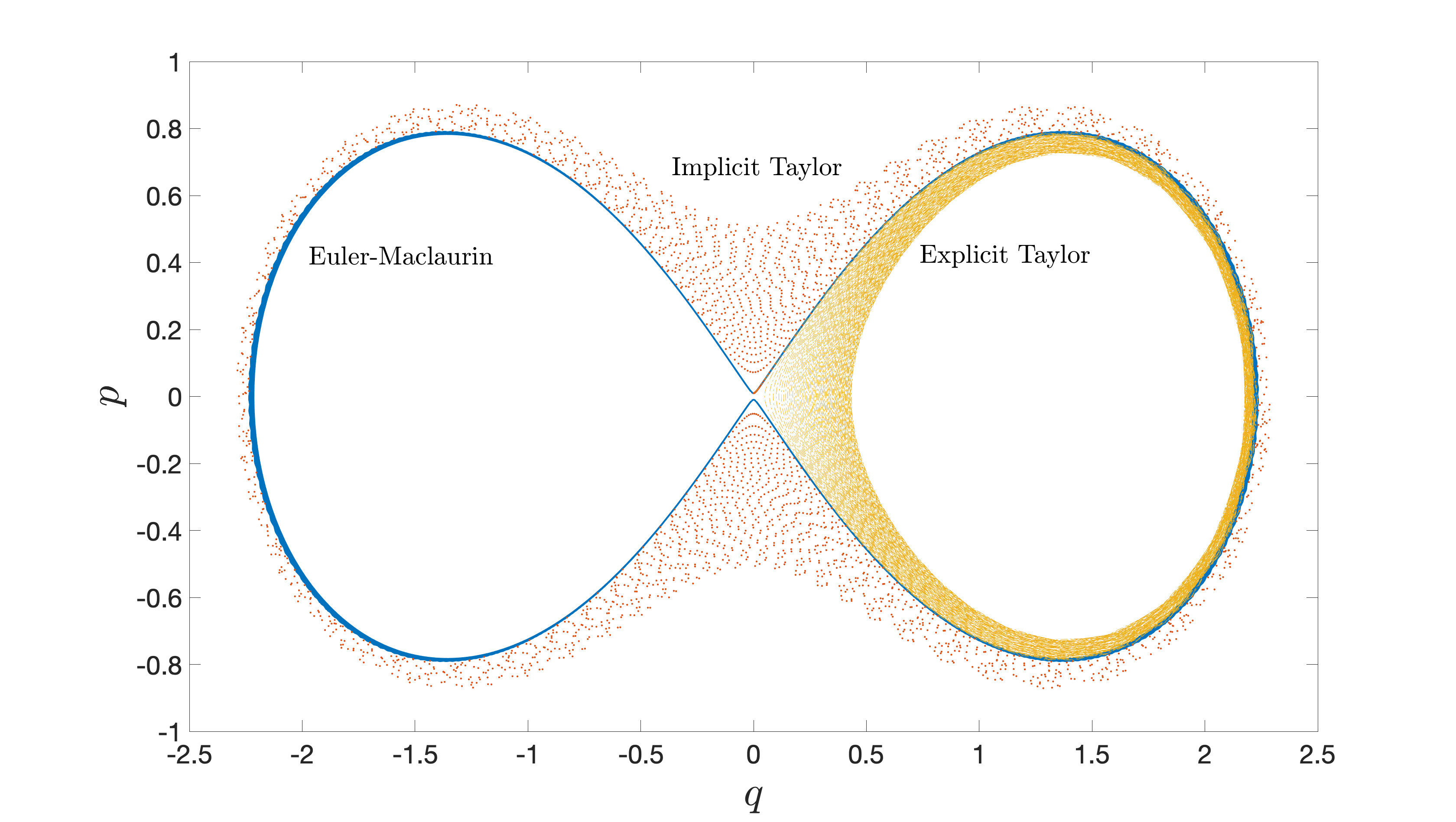}}
	\caption{Left picture: possible shapes of Cassini ovals for different values of the parameter $r$. Right picture: orbits generated by the fourth-order Euler-Maclaurin, explicit Taylor and implicit Taylor methods.}
	\label{cassini_1_fig}
\end{figure}

\begin{figure}
	\centerline{
		\includegraphics[width=8.0cm,height=4cm]{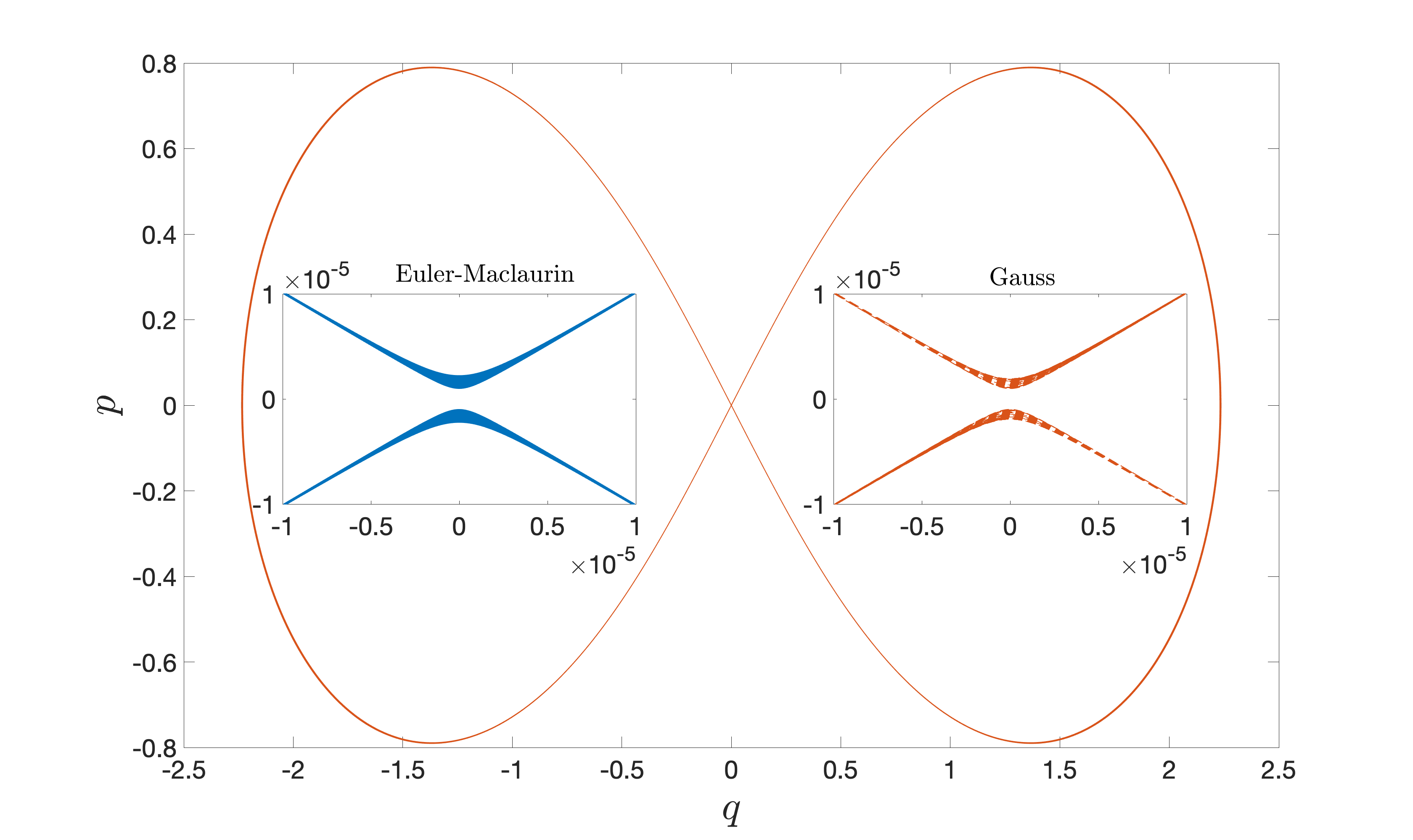} \hspace*{-.6cm}
		\includegraphics[width=8.0cm,height=4cm]{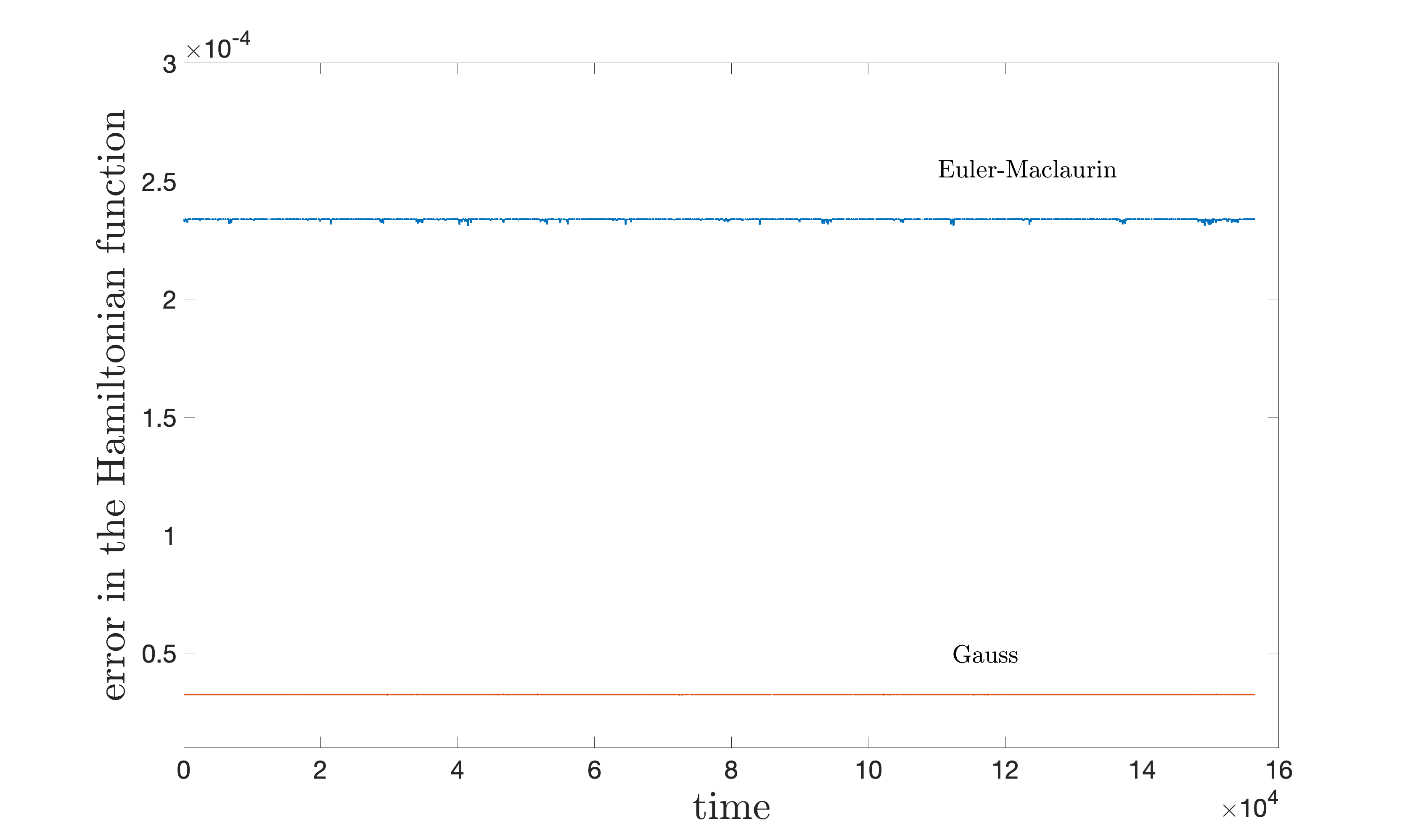}}
	\caption{Comparison between the  Euler-Maclaurin and Gauss methods of order four applied to problem (\ref{cassini_prob}
). Left picture: the orbits computed by the two integrators overlap with each other. The two close-ups show that their behavior is very similar in a neighbourhood of the origin. Right picture: error in the Hamiltonian function evaluated as the maximum absolute error in each period.}	\label{cassini_2_fig}
\end{figure}

%\subsection{H\'eon-Heiles problem}
%\label{HeHe}
%During the study of the motion of a star under the action of a gravitational potential
%of a galaxy, H\'eon and Heiles conducted a series of numerical experiments on the reduced planar Hamiltonian system defined by the following polynomial Hamiltonian function:
%\begin{equation}
%\label{HHHam}
%H(q,p)=\frac{1}{2}(p_1^2+p_2^2) + U(q,p), \qquad \mbox{with~~} U(q,p)=\frac{1}{2}(q_1^2+q_2^2)+q_1^2q_2-\frac{1}{3}q_2^3.
%\end{equation}
%The main concern was whether a further first integral, besides the total energy (\ref{HHHam}) might exist in the reduced system (see \cite{HeHe64}).\footnote{The angular momentum had been already exploited to reduce the degrees of freedom of the original system, so that it does no longer appear as a constant of motion in the reduced system defined at (\ref{HHHam}).} Starting at initial point $(q_0,p_0)$ such that $H(q_0,p_0)<1/6$, the trajectory $q(t)$ is completely contained in the equilateral triangle with vertices $P_1=(0,1)$, $P_2=(-\sqrt{3}/2,-1/2)$ and $P_3=(\sqrt{3}/2,-1/2)$. H\'enon and Heiles argued that for moderate values of the total energy (for example $H=1/12$), a further integral does indeed exist while, for higher values of the energy (such as H=1/8) the motion turns out to be chaotic. \marginpar{\tiny \red provare, in prima istanza, con questi due valori dell'Hamiltoniana.}

\section{Conclusions} \label{sec:conc}
%{%\color{blue} 
This paper studies the conservation properties of Euler--Maclaurin formulae and their implementation on the Infinity Computer. These are a family of multi-derivative one-step methods containing the classical trapezoidal method as seed formula.  The higher-order Euler--Maclaurin methods have even order $p=2s$, where $s$ denotes the maximum index of the involved derivatives of the vector field, and are topologically conjugate to a B-series symplectic formula up to the order $2s+2$. This property makes them  suitable for integrating canonical Hamiltonian systems over long times. A similar result, exploiting Theorem \ref{main_result}, has been recently derived for a class of Hermite-Obreshkov one-step methods  \cite{MS18}.

A new approach to compute the exact higher order derivatives using numerical infinities and infinitesimals is proposed. This new technique is simple, is able to work with black-box representations of the function $f(y)$ and avoids hard evaluations with tensors related to the function $f(y)$.  A comparison among this new approach and other known techniques, such as automatic differentiation, is beyond the scope of this paper and will be considered in a future work.

{\section*{Acknowledgement} We wish to thank two anonymous reviewers for careful reading our manuscript and for the valuable comments and suggestions they posted which improved the quality of the paper. We are also grateful to Ernst Hairer for helping us in fixing an issue in Theorem \ref{main_result}. }

%\section*{References}
%\bibliographystyle{elsarticle-num}
%\bibliography{euler_mclaurin}

\begin{thebibliography}{10}
	\expandafter\ifx\csname url\endcsname\relax
	\def\url#1{\texttt{#1}}\fi
	\expandafter\ifx\csname urlprefix\endcsname\relax\def\urlprefix{URL }\fi
	\expandafter\ifx\csname href\endcsname\relax
	\def\href#1#2{#2} \def\path#1{#1}\fi
	
	\bibitem{BG94}
	G.~Benettin, A.~Giorgilli, On the {H}amiltonian interpolation of near to the
	identity symplectic mappings with application to symplectic integration
	algorithms, J. Statist. Phys. 74 (1994) 1117--1143.
	
	\bibitem{HLW06}
	E.~Hairer, C.~Lubich, G.~Wanner, Geometric Numerical Integration.
	Structure-Preserving Algorithms for Ordinary Differential Equations, Second
	ed., Springer, Berlin, 2006.
	
	\bibitem{La90}
	F.~M. Lasagni, Integration methods for Hamiltonian differential equations,
	unpublished manuscript, 1990.
	
	\bibitem{HaMuSS94}
	E.~Hairer, A.~Murua, J.~Sanz-Serna, The non-existence of symplectic
	multi-derivative {Runge-Kutta} methods, BIT 34(1) (1994) 80--87.
	
	\bibitem{HZ12} E.~Hairer, C.J.~Zbinden, On conjugate symplecticity of B-series integrators, IMA J. Numer. Anal. 33(1) (2013), 57--79. 

	
	\bibitem{AmIaMaMuSe16}
	P.~Amodio, F.~Iavernaro, F.~Mazzia, M.~S. Mukhametzhanov, Y.~D. Sergeyev, A
	generalized {T}aylor method of order three for the solution of initial value
	problems in standard and infinity floating-point arithmetic, Mathematics and
	Computers in Simulation 141 (2016) 24--39.
	
	\bibitem{MaSeIaAmMu16}
	F.~Mazzia, Y.~D. Sergeyev, F.~Iavernaro, P.~Amodio, M.~S. Mukhametzhanov,
	Numerical methods for solving {ODE}s on the {I}nfinity {C}omputer, in: Proc.
	of the 2nd Intern. Conf. ``Numerical Computations: Theory and Algorithms'',
	Vol. 1776, AIP Publishing, New York, 2016, p. 090033.
	
	\bibitem{Se13}
	Y.~D. Sergeyev, Solving ordinary differential equations by working with
	infinitesimals numerically on the infinity computer, Applied Mathematics and
	Computation 219(22) (2013) 10668--10681.
	
	\bibitem{Se15}
	Y.~D. Sergeyev, Numerical infinitesimals for solving {ODE}s given as a
	black-box, in: Proc. of the International Conference on Numerical Analysis
	and Applied Mathematics 2014 (ICNAAM-2014), Vol. 1648, AIP Publishing, New
	York, 2015, p. 150018.
	
	\bibitem{SeMuMaIaAm16}
	Y.~D. Sergeyev, M.~S. Mukhametzhanov, F.~Mazzia, F.~Iavernaro, P.~Amodio,
	Numerical methods for solving initial value problems on the {I}nfinity
	{C}omputer, International Journal of Unconventional Computing 12(1) (2016)
	3--23.
	
	\bibitem{Sergeyev}
	Y.~D. Sergeyev, Arithmetic of Infinity, Edizioni Orizzonti Meridionali, CS,
	2003, 2nd ed. 2013.
	
	\bibitem{informatica}
	Y.~D. Sergeyev, A new applied approach for executing computations with infinite
	and infinitesimal quantities, Informatica 19(4) (2008) 567--596.
	
	\bibitem{Lagrange}
	Y.~D. Sergeyev, Lagrange {L}ecture: Methodology of numerical computations with
	infinities and infinitesimals, Rendiconti del Seminario Matematico
	dell'Universit\`a e del Politecnico di Torino 68(2) (2010) 95--113.
	
	\bibitem{UMI}
	Y.~D. Sergeyev, Un semplice modo per trattare le grandezze infinite ed
	infinitesime, Matematica nella Societ\`a e nella Cultura: Rivista della
	Unione Matematica Italiana 8(1) (2015) 111--147.
	
	\bibitem{EMS}
	Y.~D. Sergeyev, Numerical infinities and infinitesimals: {M}ethodology,
	applications, and repercussions on two {H}ilbert problems, EMS Surveys in
	Mathematical Sciences 4(2) (2017) 219--320.
	
	\bibitem{bolzano}
	B.~Bolzano, Paradoxien des Unendlichen, C.H. Reclam (German original), 1851.
	
	\bibitem{Trlifajova}
	K.~Trlifajov\'{a}, Bolzano's infinite quantities, Foundations of Science 23
	(2018) 681--704.
	
	\bibitem{Cantor}
	G.~Cantor, Contributions to the founding of the theory of transfinite numbers,
	Dover Publications, New York, 1955.
	
	\bibitem{Robinson}
	A.~Robinson, Non-standard Analysis, Princeton Univ. Press, Princeton, 1996.
	
	\bibitem{Levi-civita}
	T.~Levi-Civita, Sui numeri transfiniti, Rend. Acc. Lincei, Series 5a 113 (1898)
	7--91.
	
	\bibitem{Se19a}
	Y.~D. Sergeyev, Independence of the grossone-based infinity methodology from
	non-standard analysis and comments upon logical fallacies in some texts
	asserting the opposite, Foundations of Science 24(1), 153--170.
	
	\bibitem{Cococcioni}
	M.~Cococcioni, M.~Pappalardo, {Ya. D. Sergeyev}, Lexicographic multi-objective
	linear programming using grossone methodology: {T}heory and algorithm,
	Applied Mathematics and Computation 318 (2018) 298--311.
	
	\bibitem{DeLeone}
	S.~{De Cosmis}, {R. De Leone}, The use of grossone in mathematical programming
	and operations research, Applied Mathematics and Computation 218(16) (2012)
	8029--8038.
	
	\bibitem{DeLeone_2}
	R.~{De Leone}, Nonlinear programming and grossone: {Q}uadratic programming and
	the role of constraint qualifications, Applied Mathematics and Computation
	318 (2018) 290--297.
	
	\bibitem{Gaudioso&Giallombardo&Mukhametzhanov(2018)}
	M.~Gaudioso, G.~Giallombardo, M.~S. Mukhametzhanov, Numerical infinitesimals in
	a variable metric method for convex nonsmooth optimization, Applied
	Mathematics and Computation 318 (2018) 312--320.
	
	\bibitem{homogeneity}
	Y.~D. Sergeyev, D.~E. Kvasov, M.~S. Mukhametzhanov, On strong homogeneity of a
	class of global optimization algorithms working with infinite and
	infinitesimal scales, Communications in Nonlinear Science and Numerical
	Simulation 59 (2018) 319--330.
	
	\bibitem{DAlotto}
	L.~D'Alotto, Cellular automata using infinite computations, Applied Mathematics
	and Computation 218(16) (2012) 8077--8082.
	
	\bibitem{DAlotto_3}
	L.~D'Alotto, A classification of two-dimensional cellular automata using
	infinite computations, Indian Journal of Mathematics 55 (2013) 143--158.
	
	\bibitem{Margenstern_3}
	M.~Margenstern, Fibonacci words, hyperbolic tilings and grossone,
	Communications in Nonlinear Science and Numerical Simulation 21(1--3) (2015)
	3--11.
	
	\bibitem{Iudin_2}
	D.~I. Iudin, Y.~D. Sergeyev, M.~Hayakawa, Infinity computations in cellular
	automaton forest-fire model, Communications in Nonlinear Science and
	Numerical Simulation 20(3) (2015) 861--870.
	
	\bibitem{DeBartolo}
	M.~Vita, S.~D. Bartolo, C.~Fallico, M.~Veltri, Usage of infinitesimals in the
	{M}enger's {S}ponge model of porosity, Applied Mathematics and Computation
	218(16) (2012) 8187--8196.
	
	\bibitem{Caldarola_1}
	F.~Caldarola, The {S}ierpinski curve viewed by numerical computations with
	infinities and infinitesimals, Applied Mathematics and Computation 318 (2018)
	321--328.
	
	\bibitem{Koch}
	Y.~D. Sergeyev, The exact (up to infinitesimals) infinite perimeter of the
	{K}och snowflake and its finite area, Communications in Nonlinear Science and
	Numerical Simulation 31(1--3) (2016) 21--29.
	
	\bibitem{Zhigljavsky}
	A.~Zhigljavsky, Computing sums of conditionally convergent and divergent series
	using the concept of grossone, Applied Mathematics and Computation 218(16)
	(2012) 8064--8076.
	
	\bibitem{Rizza}
	D.~Rizza, Supertasks and numeral systems, in: Proc. of the 2nd Intern. Conf.
	``Numerical Computations: Theory and Algorithms'', Vol. 1776, AIP Conference
	Proceedings, New York, 2016, 090005.
	
	\bibitem{first}
	Y.~D. Sergeyev, Counting systems and the {F}irst {H}ilbert problem, Nonlinear
	Analysis Series A: Theory, Methods $\&$ Applications 72(3-4) (2010)
	1701--1708.
	
	\bibitem{Rizza19a}
	D.~Rizza, Numerical methods for infinite decision-making processes, Int. Journ.
	of Unconventional Computing 14 (2) (2019), 139-158.
	
	\bibitem{Rizza_2}
	D.~Rizza, A study of mathematical determination through {Bertrand's Paradox},
	Philosophia Mathematica 26(3) (2017) 375--395.
	
	\bibitem{Sergeyev_Garro}
	Y.~D. Sergeyev, A.~Garro, Observability of {T}uring machines: A refinement of
	the theory of computation, Informatica 21(3) (2010) 425--454.
	
	\bibitem{Sergeyev_Garro_2}
	Y.~D. Sergeyev, A.~Garro, Single-tape and multi-tape {T}uring machines through
	the lens of the {G}rossone methodology, Journal of Supercomputing 65(2)
	(2013) 645--663.
	
	\bibitem{Num_dif}
	Y.~D. Sergeyev, Higher order numerical differentiation on the {I}nfinity
	{C}omputer, Optimization Letters 5(4) (2011) 575--585.
	
	\bibitem{HNW93}
	E.~Hairer, S.~P. N{\o}rsett, G.~Wanner, Solving ordinary differential
	equations. I. Nonstiff problems., Second ed., Springer Series in
	Computational Mathematics {\bf 8}, Springer-Verlag, Berlin, 1993.
	
	\bibitem{CFM06}
	P.~Chartier, E.~Faou, A.~Murua, An algebraic approach to invariant preserving
	integrators: the case of quadratic and hamiltonian invariants, Numer. Math.
	103(4) (2006) 575--590.
	
	\bibitem{Ha08}
	E.~Hairer, Conjugate-symplecticity of linear multistep methods, J. Comput.
	Math. 26(5) (2008) 657--659.
	
	\bibitem{HaWa08}
	E.~Hairer, G.~Wanner, Analysis by its history, Undergraduate Texts in
	Mathematics. Readings in Mathematics, Springer-Verlag, New York, 2008.

    \bibitem{BFCI14} L.~Brugnano, G.~Frasca Caccia, F.~Iavernaro, Efficient implementation of Gauss collocation and Hamiltonian boundary value methods, Numer. Algorithms 65(3) (2014) 633--650.
	
	\bibitem{BI16}
	L.~Brugnano, F.~Iavernaro, Line Integral Methods for Conservative Problems,
	Monographs and Research Notes in Mathematics, CRC Press, Boca Raton, FL,
	2016.
	
	\bibitem{MS18}
	F.~Mazzia, A.~Sestini, On a class of conjugate symplectic Hermite-Obreshkov one-step methods with continuous spline extension, Axioms  7(3) (2018) 58.
		
\end{thebibliography}

\end{document}